\documentclass[10pt]{article}

\usepackage{latexsym,amsmath,amsthm}
\usepackage{times}
\usepackage{amsfonts}
\usepackage{amssymb}
\usepackage{latexsym}
\usepackage{textcomp}
\usepackage{enumerate}
\usepackage{tikz}
\usetikzlibrary{patterns,snakes}

\numberwithin{equation}{section}

\newtheorem{definition}{Definition}[section]

\newtheorem{theorem}[definition]{Theorem}
\newtheorem{lemma}[definition]{Lemma}
\newtheorem{proposition}[definition]{Proposition}
\newtheorem{corollary}[definition]{Corollary}

\newtheorem{fact}[definition]{Fact}

\usepackage[dvips]{epsfig}
\graphicspath{{./Fig_eps/}{./Eps/}}
\DeclareGraphicsExtensions{.ps,.eps}
\usepackage{color}
\usepackage{amsbsy}

\renewcommand{\geq}{\geqslant}
\renewcommand{\leq}{\leqslant}

\newcommand\id{\mathrm{id}}
\newcommand\cH{\mathcal{H}}
\newcommand{\Bin}{{\rm Bin}}
\newcommand\cols{q}
\newcommand\unif{k}
\newcommand\hyp{\cH}
\newcommand\fhyp{H}
\newcommand\hypdash{\mathcal{H}'}
\newcommand\fhypdash{H'}

\newcommand\chstar{\mathcal{H}'}

\newcommand\hdup{\hyp\textquotesingle(n,\unif,cn)}
\newcommand\hmod{\hyp (n,\unif,cn)}

\newcommand\myconst{(1.01/\unif)^{1/(\unif-1)}}

\newcommand\cbc[1]{\left\{{#1}\right\}}
\newcommand\brk[1]{\left\lbrack{#1}\right\rbrack}
\newcommand\norm[1]{\left\|{#1}\right\|}
\newcommand{\whp}{w.h.p.}

\newcommand{\rigidity}{rigidity }

\usepackage[colorlinks=true,citecolor=black,linkcolor=black,urlcolor=blue]{hyperref}

\usepackage[normalem]{ulem}

\begin{document}

\title{Rigid colourings of hypergraphs and contiguity}

\author{Peter Ayre \qquad Catherine Greenhill\thanks{Research supported by the Australian Research Council Discovery Project DP140101519.}\\
\small School of Mathematics and Statistics\\[-0.5ex]
\small UNSW Sydney\\[-0.5ex]
\small  Sydney NSW 2052, Australia\\
\small \texttt{peter.ayre@student.unsw.edu.au\qquad c.greenhill@unsw.edu.au}
}

\date{11 June 2019}

\maketitle

\begin{abstract}
We consider the problem of $q$-colouring a
 $k$-uniform random hypergraph, where $q,k \geq 3$,
and determine the rigidity threshold. 
For edge densities above the rigidity threshold, we show that almost all solutions have a linear number of 
vertices that are linearly frozen, meaning that they cannot be recoloured by a sequence of colourings that 
each change the colour of a sublinear number of vertices. 
When the edge density is below the threshold, we prove that all but a vanishing proportion of the vertices can be recoloured by a sequence of colourings that recolour only one vertex at a time.  
This change in the geometry of the solution space has been hypothesised to be the cause of the 
algorithmic barrier faced by naive colouring algorithms.  Our calculations verify predictions made by
statistical physicists using the non-rigorous cavity method.

The traditional model for problems of this type is the \emph{random colouring model}, where
a random hypergraph is chosen and then a random colouring of that hypergraph is selected.
However, it is often easier to work with the \emph{planted model}, where a random colouring is
selected first, and then edges are randomly chosen which respect the colouring.
As part of our analysis, we show that up to the condensation phase transition, 
the random colouring model is contiguous with respect to the planted model. 
This result is of independent interest.

\vspace*{\baselineskip}

\end{abstract}

\section{Introduction}
For several years, much progress made in random constraint satisfaction problems has been inspired by a highly ingenious but non-rigorous formalism from statistical physics called the \emph{cavity method}. The method predicts that for many typical constraint satisfaction problems, the geometry of the solution space undergoes dramatic changes as the constraint density increases. In particular, at a critical threshold known as the \emph{clustering threshold}, the solution space shatters into exponentially many, exponentially small clusters that are well separated. 
As we increase the density further, it is predicted \cite{Lenka} that we see the emergence of \emph{frozen} variables which take the same value for every solution within the cluster. We show that at a critical density known as the \emph{rigidity threshold}, a typical solution possesses a linear number of frozen variables. This transition is believed to mark the point where naive colouring algorithms abruptly fail to find solutions. In particular, it has been hypothesised \cite{StatPhys1, Statphys3} that it is not the clustering of solutions that causes this algorithmic barrier but rather the ridigity of the variables within the clusters.
Further, a deeper understanding of precisely what occurs at the {\rigidity threshold} has enabled the development of advanced heuristics such as survey propogation (see \cite{SurveyProp1, SurveyProp2}) that experimentally perform well even beyond the {\rigidity threshold}.

Recent work on constraint satisfaction problems has focused either on the case of $\unif$-ary variables and binary constraints (e.g. graph colouring~\cite{AminSilent, AminColouring, MolloyGraph}) or on binary variables and $\unif$-ary constraints (e.g. hypergraph $2$-colouring, see \cite{FelicCond, MolloyNSAT}) and in these cases substantial progress has been made. By contrast, relatively little is known about those problems in which both the arity of variables and the size of the constraints are greater than two. We consider one of the most natural problems of this type, namely $\cols$-colourings of a $\unif$-uniform hypergraph. In particular, we provide an adaptation of \cite{AminSilent, MolloyNSAT} to the hypergraph setting and determine the precise location of the \rigidity threshold. Further, the location we obtain coincides with predictions made by the non-rigorous cavity method (see \cite{Lenka}).

To be precise, by a $\cols$-coloring of $\fhyp=(V,E)$ we mean a map $\sigma:V\to[\cols]$ such 
that $|\sigma(e)|>1$ for all $e\in E$ (that is, no edge is monochromatic). 
For a fixed hypergraph $\fhyp$ we let $Z_\cols(\fhyp)$ denote the number of $\cols$-colourings 
of $\fhyp$.
Write $[n] = \{1,2,\ldots, n\}$.  We work in the following classical random hypergraph model. 
For a fixed $c>0$ and $m=\lfloor cn \rfloor$, let $\hyp(n,\unif,m)$ be chosen
uniformly at random from the set of all simple $\unif$-uniform hypergraphs with
vertex set $[n]$. 
For notational convenience we will often write $Z_\cols(\hyp)$ rather than $Z_\cols(\hyp(n,k,cn))$. The difference between $cn$ and $\lfloor cn \rfloor$ is negligible for large $n$ and as such will be ignored.

\textbf{Notation.}\label{Apx_notation}
We use the $O$-notation to refer to the limit $n\rightarrow\infty$. 
For example, $f(n)=O(g(n))$ means that there exists some $C>0, n_0>0$ such that for all $n>n_0$ we have $|f(n)|\leq C\cdot|g(n)|$. 
In addition, $o(\cdot), \Omega(\cdot),\Theta(\cdot)$ take their usual definitions, {except that we assume the expression $\Omega(n)$ is positive (for sufficiently large $n$) whenever we write 
$\exp(-\Omega(n))$}. We write $f(n)\sim g(n)$ if $\lim_{n\rightarrow\infty}f(n)/g(n)=1$. When discussing estimates that hold in the limit of large $\cols$ we will make 
this explicit by adding the subscript $\cols$ to the asymptotic notation. 
Therefore, $f(\cols)=O_\cols(g(\cols))$ means that there exists $C>0, \cols_0>0$ such that for all $\cols>\cols_0$ we have $|f(\cols)|\leq C\cdot |g(\cols)|$. We assume throughout that the number of vertices $n$ is sufficiently large that our estimates to hold. We say that an event holds w.h.p.\ if it holds with probability $1-o(1)$.
 
For any two $q$-colourings $\sigma,\tau$ we let $\sigma\Delta\tau$ denote the set of vertices which receive different colours under $\sigma$ and $\tau$. 

\begin{definition}
Given a $\cols$-colouring $\sigma$ of {$\fhyp$}, we say that a vertex {$v\in\fhyp$} is \emph{$\ell$-frozen} with respect to $\sigma$ if for every {$t\in\mathbb{N}$} and every sequence of $\cols$-colourings $\sigma_1,\sigma_2,\dots,\sigma_t$ such that $\vert\sigma_i\Delta\sigma_{i+1}\vert\leq \ell$ for all $i\in[t]$  we have $\sigma_t(v)=\sigma(v)$.
\end{definition}
That is, we say that $v$ is $\ell${-frozen} with respect to $\sigma$ if we cannot change the colour of $v$ by a sequence of colourings where at most $\ell$ vertices change at a time. In Theorem~\ref{IB_Freezing} 
we prove that the  value of $c$ where $O(n)$ vertices become $O(n)$-frozen is given by  
\begin{align}\label{valuecr}
c_{\rm r}=c_{\rm r}(\cols,\unif)=\frac{\cols^{\unif-1}}{\unif}\left(\ln \left[(\cols-1)(\unif -1)\right]+\ln\ln \left[(\cols-1)(\unif-1)\right] +{1+o_{\cols}(1)}\right).\end{align} 
An exact expression for the threshold is derived later: see (\ref{solution}) and (\ref{exact}).

There are further transitions in the geometry of the solution space after the point of \emph{\rigidity}. In particular, at the \emph{condensation threshold} the solution space coalesces and  clusters that are a non-vanishing proportion of the entire set of solutions begin to appear. Our results  hold up to a vanishing distance from condensation; however,
it is entirely possible that even in these larger clusters the variables are still frozen. 
The condensation threshold \cite{ACGCol} is given by
\begin{align}
c_{\text{cond}}&=c_{\text{cond}}(\cols,\unif)=(\cols^{\unif-1}-\tfrac{1}{2})\ln \cols -\ln 2-o_{\cols}(1).
\end{align}

In order to state our theorem, we need a few more definitions.
When $c>c_{\rm r}$ let $\lambda(q,k,c)=\lambda$ be the largest solution (see Section \ref{Emergenceofthecore}) to the equation 
\begin{align*}c=\frac{\cols^{\unif-1}-1}{\unif}\cdot\frac{\lambda}{(1-e^{-\lambda})^{(\cols-1)(\unif-1)}}.\end{align*}
Take a $q$-colouring $\sigma$ of a hypergraph $H$. The \emph{core} of a hypergraph with respect to $\sigma$ is the 
hypergraph formed by iteratively removing any vertex $v$ such that there exists $\ell\in[q]\backslash\{ \sigma(v)\}$ with no edge $e$ such 
that $\sigma(e\backslash\{v\})=\{\ell\}$ and $v\in e$. 
A more formal definition of the core is given in Section~\ref{Emergenceofthecore}.
For notational convenience, we define
\begin{align*}
 \Upsilon(q,k,c)=\left[\frac{(\cols^{\unif-1}-1)\cdot \lambda(q,k,c)}{ck}\right]^{\frac{1}{k-1}}
\end{align*}
We prove the following theorem  in Section \ref{Sect_FrozenCore}.

\begin{theorem}\label{IB_Freezing}
{Let $k\geq 3$. There exists  $0<\epsilon_q=o_q(1)$ and positive integer $\cols_0$ so that the following is true}. Suppose that either $\cols\geq 3$ and $c<(\cols^{\unif-1}-1)\ln \cols$, or $\cols\geq \cols_0$ and $c<c_\text{cond}-\epsilon_q$. {For a random $\cols$-colouring  $\sigma$ of $\hyp\in\cH(n,k,cn)$ } there exists $\gamma>0$ such that
\begin{enumerate}
\item[\emph{(a)}] If $c>c_{\rm r}$ then \whp$ $ the core is of size $\Upsilon(q,k,c)\cdot n+o(n)$ and 
\begin{enumerate}
\item[\emph{(i)}]   all but ${O(1)}$ vertices inside the core are  $\gamma n$-frozen with respect to $\sigma$;
\item[\emph{(ii)}] at most $o(n)$ vertices outside of the core are {$1$}-frozen  with respect to $\sigma$.
\end{enumerate}
\item[\emph{(b)}]{ If $c<c_{\rm r}$ then \whp$ $ at most $o(n)$ vertices
	are  $1$-frozen with respect to $\sigma$. }
\end{enumerate}
\end{theorem}
The value of $\epsilon_q$ in the above is set in \cite[Theorem 1.1]{ACGCol}.

The preceding theorem is the first result on hypergraph colouring to give a rigorous description of what happens around the rigidity threshold. 
	In particular, we rigorously confirm the statistical physics predictions of Gabri\'{e} et al. \cite[(67), $m=1$]{Lenka} as to the exact location of the rigidity threshold. These predictions are made using a highly ingenious heuristic called the  \emph{cavity method} which has motivated much recent work in random CSPs.	
	Rigorous results have previously only been established in the graph case ($k=2$, Molloy \cite{MolloyGraph}) and hypergraph $2$-colouring  ($q=2$, Molloy and Restrepo \cite{MolloyNSAT}). 
It is worth noting that what \cite{MolloyGraph,MolloyNSAT} refer to as the \emph{freezing threshold}, we refer to as the \emph{ridigity threshold}. We have chosen to follow the terminology used by the statistical physics community, as described by Gabri\'{e} et al. \cite{Lenka}:
	\begin{quote}
	\emph{The partition of the set of solutions into clusters allows to define the notion of the frozen variables of a solution: these are the variables which take the same color
	in all the solutions of the corresponding cluster. One can then further refine the description of the clustered phase and introduce two new phase transitions: the rigidity transition, denoted $c_{\rm r}$, above which the typical solutions contain an extensive number of frozen variables, and the freezing transition $c_{\rm f}$, above which all solutions have this
	property.}
	\end{quote}

When studying a random constraint satisfaction problem, we analyse the set of pairs 
$\pmb{(\Lambda,\sigma)}$ of constraints and satisfying assignments. 
The simplest way to develop a probabilistic model on these pairs is to first generate a random 
instance $\Lambda$, subject to the condition that it has a satisfying assignment, and then 
choose a satisfying assignment $\sigma$ of $\Lambda$ uniformly at random. 
This probabilistic model is known as the \emph{random assignment model}. 
An alternate probabilistic model, known as the \emph{planted model}, is often preferred, since 
calculations in this model are significantly easier to perform. 
In the planted model, a random assignment $\sigma$ is first generated, and then a constraint instance $\Lambda$ is randomly chosen, conditional on $\sigma$ being a satisfying assignment for $\Lambda$. 

Formally, let  $\pmb{(H,\sigma)}$ be the set of all pairs $(H,\sigma)$ of $k$-uniform hypergraphs $H$ with $cn$ edges and $q$-colourings $\sigma$ of $H$. Then we define the \emph{random colouring model}, denoted $(\pi_{\cols,\unif,n,cn}^\text{rc} )_{n\geq 1}$, as a probability distribution on  $(\pmb{ H,\sigma})$ defined by 
	\begin{align*}
	\pi_{\cols,\unif,n,cn}^\text{rc} (H,\sigma)=\left[Z_q(H) \, \binom{\binom{n}{k}}{cn}\,
   \mathbb{P}\left[\hyp\text{ is $q$-colourable}\right]\right]^{-1}.
	\end{align*}
	This is also described by the following process:
	\begin{enumerate}[$\hspace{30pt}\bullet$]
	\item[\textbf{RC1:}] Choose a $k$-uniform hypergraph $H$ with $cn$ edges uniformly at random such that $Z_q(H)>0$.
		\item[\textbf{RC2:}] Choose a $q$-colouring $\sigma$ of $H$ uniformly at random from the set of proper $q$-colourings of $H$.
	\end{enumerate}
Alternatively, define the \emph{planted model}, denoted $(\pi_{\cols,\unif,n,cn}^\text{pl} )_{n\geq 1}$, as a probability distribution on  $(\pmb{ H,\sigma})$ by setting
	\begin{align*}
	\pi_{\cols,\unif,n,cn}^\text{pl} (H,\sigma)=(1+o(1))\left[\binom{\binom{n}{k}}{cn}\, q^n\,\mathbb{P}\left[\sigma \text{ is $q$-colouring of }\hyp\right]\right]^{-1}.
\end{align*}
Take a map $\sigma:[n]\mapsto[q]$ and let  
$\mathcal{F}(\sigma):=\sum_{i=1}^q \binom{|\sigma^{-1}(i)|}{k}$ 
denote the number of monochromatic edges in the complete $k$-uniform hypergraph. The $1+o(1)$ term above arises because a vanishing proportion of maps $\sigma:[n]\mapsto[q]$ fail to meet the condition  $\mathcal{F}(\sigma)\leq \binom{n}{k}-cn$ (that is, no hypergraph exists that is coloured by $\sigma$). The planted model corresponds to the following process:
\begin{enumerate}[$\hspace{30pt}\bullet$]
\item[\textbf{PL1:}] Choose a map $\sigma:[n]\mapsto[q]$ uniformly at random subject to the condition that $\mathcal{F}(\sigma)\leq \binom{n}{k}-cn$. 
\item[\textbf{PL2:}] Generate a $k$-uniform hypergraph $H$ with $cn$ edges that are not monochromatic under $\sigma$.
\end{enumerate}
The two models clearly differ. In the planted model, a hypergraph is chosen with probability proportional to the number of colourings it has, whereas in the random colouring model the hypergraph is chosen uniformly at random.

Achlioptas and Coja-Oghlan~\cite{AminBarriers} showed that under certain circumstances, results may be transferred 
from the planted model to the random assignment model. In this and similar arguments (dubbed ``quiet planting''
in~\cite{KZ}), 
to show that an event holds with vanishing probability in the random assignment model,
it was necessary to show that it holds with at most exponentially small probability in the planted model.
This requirement has caused great technical difficulty in existing work (e.g. \cite{MolloyGraphInitial}). 
In their work on graph colourings,
Bapst et al.~\cite{AminSilent} provided a new method for comparing the random assignment model and the
planted model, namely, by proving that the two models are \emph{contiguous}.
Formally, given two sequences of probability measures $\textbf{$\pi$}=(\pi_n)_{n\geq 1}, \textbf{$\tau$}=(\tau_n)_{n\geq 1}$  defined on the same sequence of {$\sigma$-algebras}, we say that \textbf{$\pi$} is \emph{contiguous} with respect to \textbf{$\tau$} (that is, $\pi$ $\triangleleft$ $\tau$ in symbols) if for any sequence of events $(\mathcal{E}_n)_{n\geq 1}$ such that $\lim_{n\rightarrow \infty}\tau_n(\mathcal{E}_n)=0$ we have $\lim_{n\rightarrow \infty}\pi_n(\mathcal{E}_n)=0$. 
Once the random assignment model has been shown to be contiguous with respect to the planted model, it is then
only necessary to prove that a bad event has probability $o(1)$ in the planted model (not necessarily exponentially
small probability).  Bapst et al.~\cite{AminSilent} refer to this approach as ``silent planting''  and assert
that this strategy should translate to other constraint satisfaction problems.  

The second contribution of this paper is to show that this is indeed the case for hypergraph colourings.  We prove that the random colouring model is contiguous with respect to the planted model in the case of hypergraph colourings.  This result is of independent interest and may aid with the broader study of the colouring problem in the hypergraph setting.
Having established contiguity, we are then able to give a fairly direct proof of Theorem~\ref{IB_Freezing}.

The planted model is a good approximation to the random colouring model until the point of condensation. 
Our strategy is to show that the number of colourings $Z_\cols$ of the random hypergraph $\hyp\in\hyp(n,k,cn)$ is concentrated for densities below the point of condensation.

\begin{theorem}\label{IB_RVs}
Fix $\unif\geq 3$.  There exists $ \epsilon_q=o_q(1)$ and positive integer $q_0$ so that the following is true. Suppose that either $\cols\geq 3$ and $c<(\cols^{\unif-1}-1)\ln \cols$, 
or $\cols\geq \cols_0$ and $c<c_\text{cond}-\epsilon_q$. Then \begin{align*}
 \lim_{\omega\rightarrow\infty}\lim_{n\rightarrow\infty}\mathbb{P}\left[\, \vert\ln Z_\cols(\hyp)-\ln \mathbb{E}[Z_\cols(\hyp)]\, \vert\leq \omega\right]=1.\end{align*}
\end{theorem}

By applying small subgraph conditioning~\cite{SSCJanson, SSCWormald} we prove that the \emph{random colouring model}, denoted $(\pi_{\cols,\unif,n,cn}^\text{rc} )_{n\geq 1}$, is \emph{contiguous} with respect to the \emph{planted model},
denoted $(\pi_{\cols,\unif,n,cn}^\text{pl} )_{n\geq 1}$. 

\begin{theorem}\label{IB_Contiguity}
Fix $\unif\geq 3$.   There exists $ \epsilon_q=o_q(1)$ and positive integer $q_0$ so that the following is true. Suppose that either $\cols\geq 3$ and $c<(\cols^{\unif-1}-1)\ln \cols$, or $\cols\geq \cols_0$ and $c<c_\text{cond}-\epsilon_q$. Then  
\begin{align*}
(\pi_{\cols,\unif,n,cn}^\text{rc} )_{n\geq 1}\triangleleft(\pi_{\cols,\unif,n,cn}^\text{pl} )_{n\geq 1}.
\end{align*}
\end{theorem}

The structure of the paper is as follows.  In Sections~\ref{Sect_Outline}-\ref{Sect_Contiguity} we establish 
contiguity, proving Theorem \ref{IB_RVs} and Theorem \ref{IB_Contiguity}. 
In Section \ref{Sect_FrozenCore} we prove Theorem \ref{IB_Freezing}.

\section{Related work}

The planted model (``quiet planting'') provided the foundation for the study of the geometry of the solution set of several constraint satisfaction problems~\cite{AminCond, MolloyGraphInitial, MolloyGraph, MolloyNSAT}. 
The silent planting technique introduced by Bapst et al.~\cite{AminSilent} has been applied to graph colourings,
and now to hypergraph colourings. Rassmann~\cite{FR} gives more detail on the asymptotic number of graph
colourings~\cite{FR}.

The initial non-rigorous but mathematically sophisticated analysis of the \rigidity threshold was conducted by the statistical physics community in several papers, most notably \cite{Statphys4, Statphys2, StatPhys1,  Statphys3}. Achlioptas and Ricci-Tersenghi \cite{AchiRicci} were the first to establish rigorously that \rigidity occurs in a random constraint satisfaction problem. They showed for a significant range of edge densities below the satisfiability threshold in random $k$-SAT, a significant proportion of variables were $1$-frozen.
	
The geometry of the solution space  for $k$-colourability of graphs was originally studied by
Mulet et al.\ in~\cite{MuletPagani}. 
The appearance of frozen variables
was studied in \cite{AminBarriers} in several CSPs including $k$-SAT, graph colouring, and hypergraph $2$-colouring. 
Molloy \cite{MolloyGraphInitial} provided the first rigorous analysis of 
the \rigidity  threshold.  More recently, Molloy has updated his analysis~\cite{MolloyGraph} by applying silent 
planting, yielding stronger results. The updated work contains the most current results on the \rigidity threshold
for graph colourings. 

Molloy and Restrepo~\cite{MolloyNSAT} studied a a broad range of boolean constraint 
satisfaction problems, including hypergraph $2$-colouring. 
They provided a sophisticated analysis of the geometry of the solution
space and determined the \rigidity threshold. 
The process undertaken in \cite{MolloyNSAT} forms the basis of our approach for $q$-colouring 
$k$-uniform hypergraphs.  However,~\cite{MolloyNSAT} does not consider 
CSPs that allow more than two values for any variable.
To the best of our knowledge, our work is the first to rigorously study the rigidity threshold for a random CSP model where the arity of the variables and the size of constraints may both be greater than two.

As mentioned earlier, Semerjian~\cite{Statphys2} and Zdeborov{\' a}~\cite{StatPhys1} conjectured that the rigidity
threshold is the cause of the ``algorithmic barrier'' experienced by naive colouring algorithms.
For a study of the Glauber dynamics on hypergraph $q$-colourings below the clustering threshold, see Anastos and Frieze~\cite{Glauber}.

The freezing threshold occurs above the rigidity threshold and marks the point at which all solutions contain
``an extensive number of frozen variables''~\cite{Lenka}. 
Braunstein et al.~\cite{bicolhyp} analytically estimated the freezing threshold for the problem of 2-colouring hypergraphs.

\section{Outline of contiguity argument}\label{Sect_Outline}

As in \cite{AminSilent} we use the following version of  small subgraph conditioning.
\begin{theorem}\label{SSC}
\emph{\cite{SSCJanson, SSCWormald}} Suppose that $(\delta_\ell)_{\ell\geq 2}$, $(\lambda_\ell)_{\ell\geq 2}$ are sequences of real numbers such that $\delta_\ell \geq -1$ and $\lambda_\ell>0$ for all $\ell$. Assume that $(C_{\ell,n})_{\ell\geq 2, n\geq 1}$ and $(Z_n)_{n\geq 1}$ are random variables such that each $C_{\ell,n}$ takes values in the non-negative integers. Additionally, suppose that for each $n$ the random variables $C_{2,n},\dots,C_{n,n}$ and $Z_n$ are defined on the same probability space. Let $(X_\ell)_{\ell\geq 2}$ be a sequence of independent random variables such that $X_\ell$ has distribution {\rm Po}$(\lambda_\ell)$. Assume that the following four conditions hold:
\begin{enumerate}
	\item [\emph{\textbf{SSC1}}] For any integer $L\geq 2$ and any integers $x_2,\dots,x_L\geq 0$ we have
	\begin{align*}
		\lim_{n\rightarrow\infty}\mathbb{P}[C_{\ell}=x_\ell \,\, \forall\text{ }\ell=2,\dots,L]=\prod_{\ell=2}^L\mathbb{P}[X_\ell=x_\ell].
	\end{align*}
	\item [\emph{\textbf{SSC2}}] For any integer $L\geq 2$ and any integers $x_2,\dots,x_L\geq 0$ we have
	\begin{align*}
		\lim_{n\rightarrow\infty}\frac{\mathbb{E}[Z_n\vert C_{\ell}=x_\ell \,\, \forall\text{ }\ell=2,\dots,L]}{\mathbb{E}[Z_n]}=\prod_{\ell=2}^L(1+\delta_\ell)^{x_\ell}\exp(-\lambda_\ell\delta_\ell).
	\end{align*}
	\item [\emph{\textbf{SSC3}}] We have $\sum_{\ell=2}^\infty \lambda_\ell\delta^2_\ell<\infty$.
	\item [\emph{\textbf{SSC4}}] We have $\lim_{n\rightarrow\infty}\mathbb{E}[Z^2_n]/\mathbb{E}[Z_n]^2\leq \exp[\sum_{\ell=2}^\infty\lambda_\ell\delta_\ell^2]$.
\end{enumerate}
Then the sequence $(Z_n/\mathbb{E}[Z_n])_{n\geq 1}$ converges in distribution to $\prod_{\ell=2}^\infty(1+\delta_\ell)^{X_\ell}\exp(-\lambda_\ell\delta_\ell)$.
\end{theorem}

It will be technically more convenient to work with an alternate random hypergraph model 
$\hdup$. Here $\hypdash\in\hdup$ is a (multi-)hypergraph on the vertex set $[n]$ with $cn$ edges, where the edge set is chosen uniformly at random from the set of all $k$-subsets of $[n]$ 
\emph{with replacement}. We will use the random model $\hdup$ in all calculations, and only 
return to $\hmod$ in the proofs of Theorems~\ref{IB_Freezing} and~\ref{IB_RVs}.
As with $\hmod$, we  often refer to $Z_\cols(\hdup)$ as $Z_\cols(\hypdash)$. 

Instead of working with the number of $\cols$-colourings  we instead focus our attention on  colourings that are appropriately \emph{balanced}. This greatly reduces the complexity of our arguments while making little difference to the estimates. More precisely, for a {map} $\sigma:[\cols]\rightarrow[n]$ we define the \emph{colour density} $\rho(\sigma)=(\rho_1(\sigma),\rho_2(\sigma),\dots,\rho_\cols(\sigma))$ where $\rho_i(\sigma)=n^{-1}\vert\sigma^{-1}(i)\vert$ for all $i\in [\cols]$. Further, for a given hypergraph $H$ let $Z_{\cols,\rho}(H)$ be the number of $\cols$-colourings of $H$ with colour density $\rho$. Let $\mathcal{C}_\cols(n)$ denote the set of all possible colour densities and $\rho^{\star}=(1/\cols,\dots1/\cols)$ be the central density. Throughout this text we let $\omega=\omega(n)$ be any function of $n$ such that $\lim_{n\rightarrow\infty}\omega(n)=\infty$ arbitrarily slowly. Further, we say that a map $\sigma$ is $(\omega,n)$-\emph{balanced} if 
\begin{align*}
{\vert\rho_i(\sigma)-\cols^{-1}\vert\leq \omega^{-1}n^{-1/2}\hspace{0.2cm}\text{ for all }\hspace{0.2cm}i\in[\cols].}
\end{align*} 
Let $\mathcal{B}_{n,\cols}(\omega)$ be the set of all $(\omega,n)$-\emph{balanced} maps. Finally, let $Z_{\cols,\omega}({\fhypdash})$ be the number of $(\omega,n)$-\emph{balanced} $\cols$-colourings of the hypergraph ${\fhypdash}$. We will prove the following proposition in Section \ref{SubSect_TFM}.
\begin{proposition}\label{FM_FinalCalc}
	We have
	\begin{align*}
	\mathbb{E}[Z_{\cols,\omega}({\hypdash})]\sim(2\pi n)^{\frac{1-\cols}{2}}\cols^{\cols/2}\vert\mathcal{B}_{n,\cols}(\omega)\vert\exp\left\{n\ln q+cn\ln\left(1-\cols^{1-\unif}\right)+\frac{c\unif(\unif-1)}{2}\left({\frac{\cols-1}{\cols^{\unif-1}-1}}
	\right)\right\}.
	\end{align*}
	In particular, $\ln \mathbb{E}[Z_{q,\omega}(\hypdash)]=\ln \mathbb{E}[Z_{q}(\hypdash)]+O(\ln\omega)$.
\end{proposition}

The basic strategy is to show that the fluctuations in $Z_{\cols,\omega}({\hypdash})$ can be attributed to fluctuations in the number of loose short cycles in the hypergraphs. 
More specifically, a loose cycle of length $\ell$ is a set of edges $\{ e_0,e_1,\dots,e_{\ell-1}\}$ such that 
\[ |e_i\cap e_j| = \begin{cases} 1 & \text{ if $i-j\equiv \pm1\pmod{\ell}$},\\
                      0 & \text{ otherwise.}
\end{cases}
\] 
Let $C_{\ell,n}({\hypdash})$ be the number of loose cycles of length $\ell$ in  ${\hypdash}$. For $\ell\geq 2$, define 
\begin{align}\label{lambdadelta}
	\lambda_\ell=\frac{[c\unif(\unif-1)]^\ell}{2\ell}\hspace{0.5cm}\text{and}\hspace{0.5cm}\delta_\ell=\frac{(-1)^\ell(\cols-1)}{{(\cols^{\unif-1}-1)^\ell}}.
\end{align}
The next lemma, proved in Section \ref{SubSect_CLC}, shows that the random variables $C_{\ell,n}$ are asymptotically independent Poisson.

\begin{lemma}\label{PoissonDist}
If $x_2,\dots,x_L$ are non-negative integers then
\begin{align*}
\lim_{n\rightarrow\infty}\mathbb{P}\left[C_{\ell,n}=x_\ell\,\, \forall\text{ }\ell=2,\dots,L\right]=\prod_{\ell=2}^L\mathbb{P}[{\rm Po}(\lambda_\ell)=x_\ell].
\end{align*}
\end{lemma}

In Section \ref{SubSect_CLC} we investigate the impact of cycle counts on the first moment of $Z_{\cols,\omega}({\hypdash})$, proving the following.

\begin{proposition}\label{Res_CycleImpact}
	Assume that $\cols\geq 3$ and $c>0$. Let $\omega(n)>0$ be any sequence such that $\lim_{n\rightarrow\infty}\omega(n)\rightarrow\infty$. If $x_2,\dots,x_\ell$ are non-negative integers then 
	\begin{align}
		\frac{\mathbb{E}[Z_{\cols,\omega}({\hypdash})\mid C_{\ell}=x_\ell\hspace{0.1cm}\forall\hspace{0.1cm} 2\leq \ell\leq L]}{\mathbb{E}[Z_{\cols,\omega}({\hypdash})]}\sim\prod_{\ell=2}^L(1+\delta_\ell)^{x_\ell}\exp(-\lambda_\ell\delta_\ell),
	\end{align}
{and the sequences $(\lambda_{\ell})_{\ell\geq 2},(\delta_\ell)_{\ell\geq 2} $ satisfy $
		\sum_{\ell=2}^\infty \lambda_\ell\delta_\ell^2<\infty.$}
\end{proposition}

We also need to have very precise estimates of the second moment of $Z_{\cols,\omega}({\hypdash})$. Unfortunately the second moment cannot be uniformly analysed for the edge density range we require. Instead we have to divide our analysis into two distinct regimes. This division is caused by much the same reason that separates the first (see \cite{AchNaor,CathHyp}) and second (see \cite{ACGCol,AminColouring}) generation arguments for lower bounds on the $\cols$-colourability threshold. In the first case, a (relatively) simple and carefully executed second moment argument will yield the required estimates. In the second, an alternate random variable $\widetilde{Z}_{\cols,\omega}$ is used in the second moment arguments of \cite{ACGCol,AminColouring} {to extend the density range to near the condensation threshold}. In Section \ref{SubSect_TSM} we show the following.

\begin{proposition}\label{Res_SimpleSecond}
	Assume that $\cols\geq 3$ and $c<(\cols^{\unif-1}-1)\ln \cols$. Then
	\begin{align*}
		\frac{
			\mathbb{E}[Z_{\cols,\omega}({\hypdash})^2]
		}{
			\mathbb{E}[Z_{\cols,\omega}({\hypdash})]^2	
		}\sim
		\exp\left\{\sum_{\ell=2}^\infty\lambda_\ell\delta_\ell^2\right\}.
	\end{align*}
\end{proposition}

\begin{proposition}\label{Res_HardSecond}
	Assume that $\cols\geq \cols_0$.   { There exists positive $ \epsilon_q=o_q(1)$ such that} for $(\cols^{\unif-1}-1)\ln \cols\leq c<c_{\text{cond}}-\epsilon_q$, there exists an integer-valued random variable $0\leq {\widetilde Z}_{\cols,\omega}\leq Z_{\cols,\omega}$ such that
	\begin{align*}
		\mathbb{E}\left[{\widetilde Z}_{\cols,\omega}({\hypdash})\right]\sim\mathbb{E}\left[Z_{\cols,\omega}({\hypdash})\right],\hspace{0.5cm}\text{and}
		\hspace{1cm} \frac{
			\mathbb{E}[\widetilde Z_{\cols,\omega}({\hypdash})^2]
		}{
			\mathbb{E}[\widetilde Z_{\cols,\omega}({\hypdash})]^2	
		}\leq(1+o(1))
		\exp\left\{\sum_{\ell=2}^\infty\lambda_\ell\delta_\ell^2\right\}.
	\end{align*}
\end{proposition}

Fortunately we do not need to repeat the analysis of the impact of cycle counts for the new random variable $\widetilde{Z}_{\cols,\omega}$. 
It has been shown in the graph case \cite[Corollary 2.6]{AminSilent} that what we need follows directly from Proposition \ref{Res_CycleImpact} and Proposition \ref{Res_SimpleSecond}. The analogous result for the hypergraph setting is
Corollary~\ref{AminRep1}, stated below. The proof is omitted, as it is identical to the proof in the graph case~\cite[Corollary 2.6]{AminSilent}. We simply remark that the proof relies on Lemma \ref{PoissonDist}, Proposition \ref{Res_CycleImpact} and Proposition \ref{Res_HardSecond}.

\begin{corollary}\label{AminRep1}
\cite[Corollary 2.6]{AminSilent}
For $x_2,\dots,x_\ell$ be non-negative integers. With the assumptions and notation of Proposition \ref{Res_HardSecond} we have
\begin{align*}
		\frac{\mathbb{E}[{\widetilde Z}_{\cols,\omega}({\hypdash})\mid C_{\ell,n}=x_\ell\,\forall\, 2\leq \ell\leq L]}{\mathbb{E}[{\widetilde Z}_{\cols,\omega}({\hypdash})]}\sim\prod_{\ell=2}^L(1+\delta_\ell)^{x_\ell}\exp(-\lambda_\ell\delta_\ell).
	\end{align*}
	\end{corollary}

As previously, Corollary \ref{AminCoro2_7} follows in a similar fashion to the graph case, proved by Bapst et al.~\cite[Corollary 2.7]{AminSilent}. We omit the proof but note that it relies on Lemma \ref{PoissonDist}, Propositions \ref{FM_FinalCalc}$\,$--$\,$\ref{Res_HardSecond} and Corollary \ref{AminRep1}.

\begin{corollary}\label{AminCoro2_7}\cite[Corollary 2.7]{AminSilent}
Fix $\unif\geq 3$. There exists positive $\epsilon_q=o_q(1)$ and positive integer $q_0$ so that the following is true. Suppose that either $\cols\geq 3$ and $c<(\cols^{\unif-1}-1)\ln \cols$, or $\cols\geq \cols_0$ and $c<c_\text{cond}-\epsilon_q$. Then
	\begin{align}\label{coro27eqn}
		{\lim_{\epsilon\rightarrow0}\lim_{n\rightarrow\infty}\mathbb{P}\left[\frac{Z_\cols(\hypdash)}{\mathbb{E}[Z_\cols(\hypdash)]}\geq \epsilon\right]=1.}
	\end{align}
\end{corollary}

\begin{proof}[Proof of Theorem {\ref{IB_RVs}} 
	{(see \cite[Theorem 1.1]{AminSilent})}]
We claim that 
	\begin{align}\label{distofCol}
{		\lim_{\omega\rightarrow\infty}\lim_{n\rightarrow\infty}\mathbb{P}[\vert\ln Z_\cols(\hypdash)-\ln \mathbb{E}[Z_\cols(\hypdash)]\vert<\omega]=1.}
	\end{align}
To see that this is the case, note that Corollary~\ref{AminCoro2_7} implies that
\begin{align*}
\lim_{\omega\rightarrow\infty}\lim_{n\rightarrow\infty}\mathbb{P}[\ln Z_\cols(\hypdash)-\ln \mathbb{E}[Z_\cols(\hypdash)]>-\omega]=1.
\end{align*}
where $\omega=-\ln\epsilon$.  Further, by Markov's inequality,
\begin{align*}
\lim_{\omega\rightarrow\infty}\lim_{n\rightarrow\infty}\mathbb{P}[\ln Z_\cols(\hypdash)-\ln \mathbb{E}[Z_\cols(\hypdash)]<\omega]=1,
\end{align*} 
To derive Theorem \ref{IB_RVs} from the above, let $S$ be the event that $\hypdash$ consists of $cn$ distinct edges. The hypergraph $\hypdash(n,k,cn)$, conditional on $S$, is identical to $\hyp(n,k,cn)$. Since
	$\mathbb{P}[S]\geq
		1- O(n^{2-k}),
	$ 
it follows from (\ref{distofCol}) that
\begin{align}\label{ConcNoDup}
1&=	
\lim_{\omega\rightarrow\infty}\lim_{n\rightarrow\infty}\mathbb{P}[\vert\, \ln Z_\cols(\hypdash)-\ln \mathbb{E}[Z_\cols(\hypdash)]\, \vert<\omega\mid S]\nonumber
\\
&=\lim_{\omega\rightarrow\infty}\lim_{n\rightarrow\infty}\mathbb{P}[\vert\ln Z_\cols(\hyp)-\ln \mathbb{E}[Z_\cols(\hypdash)]\vert<\omega\mid S].
\end{align}
Furthermore, we know from \cite[Lemma 3.2]{ACGCol} that 
\begin{align}\label{FM_Colourings}
\mathbb{E}[Z_q(\hyp)]=\Theta(q^n(1-q^{1-k})^{cn}).
\end{align} 
Combining this with Proposition \ref{FM_FinalCalc} gives $\mathbb{E}[Z_q(\hypdash)]=\Theta(\mathbb{E}[Z_q(\hyp)])$ and so it follows from (\ref{ConcNoDup}) that
\begin{align*}
{			\lim_{\omega\rightarrow\infty}\lim_{n\rightarrow\infty}\mathbb{P}[\vert\ln Z_\cols(\hyp
			)-\ln \mathbb{E}[Z_\cols(\hyp)]\vert<\omega]=1.}
\end{align*}
This completes the proof.
\end{proof}

We conclude this section with a proof of Theorem \ref{IB_Contiguity}. This proof is very similar to that given in the graph case (see \cite[Theorem 1.2]{AminSilent}) but is included here for completeness.

\begin{proof}[Proof of Theorem {\ref{IB_Contiguity}}]
Assume for a contradiction that $(\mathcal{A}_n)_{n\geq 1}$ is a sequence of events on the set of pairs $(H,\sigma)$ such that for some fixed number $0<\epsilon<1/2$ we have
\begin{align}\label{intermed0}
\lim_{n\rightarrow\infty}\pi^{\text{pl}}_{q,k,n,cn}=0\hspace{1cm}\text{while}\hspace{1cm}\limsup_{n\rightarrow\infty}\pi^{\text{rc}}_{q,k,n,cn}>\epsilon.
\end{align}
Let $\hyp(n,k,cn,\sigma)$ denote a  $k$-uniform hypergraph on $[n]$ with precisely $cn$ distinct edges, such that no edge is monochromatic under $\sigma$, chosen uniformly at random. Then
\begin{align}\label{intermed1}
& \mathbb{E}[Z_q(\hyp(n,k,cn))\textbf{1}_{\mathcal{A}_n}]\\
	&=\sum_{\sigma:[n]\rightarrow[\cols]}\mathbb{P}[\sigma \text{ is a $q$-colouring of $\hyp(n,k,cn)$ and $(\hyp(n,k,cn),\sigma)\in\mathcal{A}_n$}]\nonumber \\
  &=\sum_{\sigma:[n]\rightarrow[\cols]}\mathbb{P}[(\hyp(n,k,cn),\sigma)\in\mathcal{A}_n\mid\sigma \text{ is a $q$-colouring of $\hyp(n,k,cn)$}]\cdot\mathbb{P}[\sigma\text{ is a $q$-colouring of $\hyp(n,k,cn)$}]\nonumber \\
  &=\sum_{\sigma:[n]\rightarrow[\cols]}\mathbb{P}[\hyp(n,k,cn,\sigma)\in\mathcal{A}_n]\cdot\mathbb{P}[\sigma\text{ is a $q$-colouring of $\hyp(n,k,cn)$}]\nonumber \\
&= O((1-q^{1-k})^{cn})\cdot\sum_{\sigma:[n]\rightarrow[\cols]}\mathbb{P}[\hyp(n,k,cn,\sigma)\in\mathcal{A}_n]\nonumber \\
  &= O(q^n(1-q^{1-k})^{cn})\cdot\mathbb{P}[\hyp(n,k,cn,\sigma)\in\mathcal{A}_n]=o(q^n(1-q^{1-k})^{cn}).
\end{align}
By Corollary \ref{AminCoro2_7}, for any $\epsilon>0$ there is $\delta>0$ such that for all large enough $n$ we have
\begin{align}\label{intermed2}
	{\mathbb{P}[Z_\cols(\hyp)<\delta\,\mathbb{E}[Z_\cols(\hyp)]]<\epsilon/2.}
\end{align}
Now, let $\mathcal{E}$ be the event that $Z_\cols(\hyp)\geq \delta \, \mathbb{E}[Z_q(\hyp)]$ and let $t=\pi^{\text{rc}}_{q,n,cn}[\mathcal{A}_n\mid \mathcal{E}]$. Then
\begin{align}\label{intermed3}
	\mathbb{E}[Z_\cols(\hyp)\textbf{1}_{\mathcal{A}_n}]
	\, & \geq \delta\, \mathbb{E}[Z_q(\hyp)]\cdot \mathbb{P}[\big((\hyp,\tau)\in\mathcal{A}_n\big)\cap \mathcal{E}]\nonumber \\
&\geq \delta t\, \mathbb{E}[Z_q(\hyp)]\cdot \mathbb{P}[\mathcal{E}]\nonumber \\
&\geq \delta t\epsilon \,\,\mathbb{E}[Z_q(\hyp)]/2\nonumber\\
  &=\frac{\delta t\epsilon}{2}\cdot \Omega(q^n(1-q^{1-k})^{cn}).
\end{align}
Combining (\ref{intermed1}) and (\ref{intermed3}), we obtain $t=o(1)$. Hence, (\ref{intermed2}) implies that
\begin{align*}
\pi^\text{rc}_{q,k,n,cn}[\mathcal{A}_n]=\pi^\text{rc}_{q,k,n,cn}[\mathcal{A}_n\mid\neg\mathcal{E}]\cdot\mathbb{P}[\neg\mathcal{E}]+t\cdot\mathbb{P}[\mathcal{E}]\leq \mathbb{P}[\neg\mathcal{E}]+t\leq \epsilon/2+o(1),
\end{align*}
in contradiction to (\ref{intermed0}).
\end{proof}

\section{Contiguity}\label{Sect_Contiguity}

In this section we provide the necessary estimates for  Propositions \ref{Res_CycleImpact}$\,$--$\,$\ref{Res_HardSecond}. 
Section \ref{SubSect_TFM} will deal with the first moment and Section \ref{SubSect_CLC} the second moment. Section \ref{SubSect_TSM} is devoted to the proof of Proposition \ref{Res_SimpleSecond} and Proposition \ref{Res_HardSecond}. In much of what follows we will see calculations that resemble those of \cite{ACGCol}. However, \cite{ACGCol} only provides estimates of constant relative error, whereas we require that our calculations are precise asymptotically.

\subsection{The first moment}\label{SubSect_TFM}

\begin{proof}[Proof of Proposition \ref{FM_FinalCalc}]
 Fix an $(\omega,n)$-\emph{balanced} density $\rho\in \mathcal{C}_\cols(n)$. By definition,
\[
\mathbb{E}[Z_{\cols,\omega}({\chstar})]=\sum_{\rho\in\mathcal{B}_{n,\cols}(\omega)}\mathbb{E}[Z_{\cols,\rho}({\chstar})].
\]
Hence we first estimate $\mathbb{E}[Z_{\cols,\rho}({\chstar})]$. We know from  the independence of edges that
\begin{align*}
\mathbb{E}[Z_{\cols,\rho}({\chstar})] \sim \binom{n}{\rho_1 n,\dots,\rho_\cols n}\,
   \left(1-\binom{n}{\unif}^{-1}\, \sum_{i=1}^\cols \binom{\rho_i n}{\unif}\right)^{cn}.
\end{align*}
We will begin our calculations with the second factor. In particular, 
\begin{align*}
\binom{n}{\unif}^{-1} \sum_{i\in[\cols]} \binom{\rho_i n}{\unif}
 &= \sum_{i\in[\cols]}\prod_{j\in[\unif]}\left(\frac{\rho_in-j}{n-j}\right)
=\sum_{i\in[\cols]}\rho_i^\unif\prod_{j\in[\unif]}\left({1-\frac{j}{n\rho_i}}\right)\left({1+\frac{j}{n}+O(n^{-2})}\right) \\
&=\sum_{i\in[\cols]}\rho_i^\unif\prod_{j\in[\unif]}\left({1+\frac{j(\rho_i-1)}{n\rho_i}}{+O(n^{-2})}\right) \\
&=\sum_{i\in[\cols]} \rho_i^\unif+\frac{\unif(\unif-1)}{2n}\sum_{i\in[\cols]} \rho_i^{\unif-1}(\rho_i-1)+O(n^{-2}),
\end{align*}
and so
\begin{align*}
	\left(1-\binom{n}{\unif}^{-1}\sum_{i\in[\cols]} \binom{\rho_i n}{\unif}\right)^{cn}
	&\sim\exp\left\{cn\ln\left(1-\sum_{i\in[\cols]} \rho_i^\unif-\frac{\unif(\unif-1)}{2n}
		\sum_{i\in[\cols]} \rho_i^{\unif-1}(\rho_i-1)\right)\right\} \\
  &\sim \exp\left\{cn\ln\left(1-\sum_{i\in[\cols]} 
				\rho_i^\unif\right)+\frac{c\unif(\unif-1)}{2}\left({\frac{\sum_{i\in[\cols]} \rho_i^{\unif-1}-\sum_{i\in[\cols]} \rho_i^{\unif}}
				{1-\sum_{i\in[\cols]} \rho_i^\unif}}\right)\right\}.
\end{align*}
Since $\rho$ is $(\omega,n)$-balanced it follows that 
 $\sum_{i\in[\cols]} \rho_i^\unif\sim \cols^{1-\unif}$ and $\sum_{i\in[\cols]} \rho_i^{\unif-1}\sim \cols^{2-\unif}$. Hence
\begin{align*}
\left(1-\binom{n}{\unif}^{-1}\sum_{i\in[\cols]} \binom{\rho_i n}{\unif}\right)^{cn}
	\sim \exp\left\{cn\ln\left(1-\cols^{1-\unif}\right)+\frac{c\unif(\unif-1)(\cols-1)}{2(\cols^{\unif-1}-1)}\right\}.
\end{align*}
On the other hand, and again since $\rho$ is $(\omega,n)$-\emph{balanced}, we have
\begin{align*}
 \binom{n}{\rho_1 n,\dots,\rho_\cols n}\sim (2\pi n)^{\frac{1-\cols}{2}}\, \cols^{\cols/2}
\, \exp\left\{-n\sum_{i\in[\cols]}\rho_i\ln\rho_i\right\}.
\end{align*}
Since $\norm{\rho-\rho^{\star}}_2^2=o({n^{-1}})$ and {$\sum_{i\in[\cols]}\rho_i=1$}, a Taylor expansion of the last factor yields
\begin{align*}
{-\sum_{i\in[\cols]}\rho_i\ln\rho_i
=\ln\cols+(1-\ln \cols)\left(1-\sum_{i\in[\cols]}\rho_i\right)-\frac{\cols}{2}\norm{\rho-\rho^{\star}}_2^2=\ln\cols+o(n^{-1}).} 
\end{align*}
 The result follows by summing over all $\rho_i\in \mathcal{B}_{n,\cols}(\omega)$.
\end{proof}

\subsection{Counting loose cycles}\label{SubSect_CLC}
{For a fixed positive integer $L$,} let $x_2,\dots x_L$ denote a sequence of non-negative integers. Further, let $S$ be the event that $C_{\ell,n}=x_\ell$ for $\ell=2,\dots,L$, and let $\mathcal{V}(\sigma)$ be the event that $\sigma$ is a $\cols$-colouring of the random hypergraph ${\chstar}$. Recall the definition of $\lambda_{\ell},\delta_\ell$ given in (\ref{lambdadelta}).
\begin{lemma}\label{SSC2_L1}
	Let $\mu_\ell=\lambda_\ell(1+\delta_\ell)
	$. Then $\mathbb{P}[S\mid \mathcal{V}(\sigma)]\sim \prod^L_{\ell=2}\frac{\exp(-\mu_\ell)}{x_\ell!}\mu_\ell^{x_\ell}$ for any $\sigma\in\mathcal{B}_{n,\cols}(\omega)$.
\end{lemma}

\begin{proof}
 We show that for any sequence of integers $m_2,\dots,m_L\geq 0$, the joint factorial moments satisfy 
\[ \mathbb{E}[(C_{2,n})_{m_2}\dots (C_{L,n})_{m_L}\mid \mathcal{V}(\sigma)]\sim \prod_{\ell=2}^L\mu_\ell^{m_\ell}.
\] Then the Lemma follows from \cite[Theorem 1.23]{BollobasRG}. Let $Y$ {denote the number of sequences of distinct loose cycles such that the first $m_2$ have length $2$, the next $m_3$ have length $3$ and so on up to $m_L$, where we require that all cycles are vertex-disjoint.} Further, let 
$Y'$ denote the number of these sequences where two or more cycles intersect. We analyse these cases in Proposition \ref{SSC_PCycle2} and Proposition \ref{SSC_PCycle1} below.
\end{proof}

\begin{proposition}\label{SSC_PCycle2}
We have $\mathbb{E}[Y\mid \mathcal{V}(\sigma)]\sim {\prod_{\ell=2}^L \mu_\ell^{m_\ell}}$.
\end{proposition}

\begin{proof}
We follow the proof of \cite{AminSilent} with careful modification for the hypergraph setting. Let $D_\ell$ be the number of rooted, directed loose cycles of length $\ell$. Recall that loose cycles are those where one edge overlaps with the next in a single vertex. If $(v_1,\dots,v_\ell)$ are the overlapping vertices, we call $(\sigma(v_1),\dots,\sigma(v_\ell))$ the \emph{type} of the cycle under $\sigma$. For $t=(t_1,\dots,t_\ell)$ and $0\leq s \leq \ell$ we let $D_{\ell,t,s}$ be the number of rooted directed cycles of type $t$ where there are precisely $s$ elements $t_{i_1},\dots t_{i_s}$ such that $t_{i_j}=t_{i_j-1}$ for all $1\leq j\leq s$. We claim that
\begin{align}\label{Cycle_typeest}
	\textbf{E}[D_{\ell,t,s}\mid \mathcal{V}(\sigma)]&\sim \left(\frac{n}
				{\cols}\right)^\ell\cdot(cn)^{\ell}\cdot\left[\frac{\binom{n}{\unif -2}}{ \binom{n}{\unif}-q\binom{n/\cols}{\unif}}\right]^{\ell-s}\cdot
\left[\frac{\binom{n}{\unif -2}- \binom{n/\cols}{\unif-2}}
{ \binom{n}{\unif}-q\binom{n/\cols}{\unif}} \right]^{s}
			\sim \left[ \frac{c\unif(\unif-1)}{\cols-\cols^{2-\unif}}\right]^{\ell}(1-\cols^{2-\unif})^s.
 \end{align} 
In the above the first factor is asymptotic for the number of ways to choose $\ell$ vertices of colour $t_i$, the second is asymptotically equal to the number of ways to choose a sequence of $\ell$ edges, the third is the probability that a vertex pair $(v_i,v_{i+1})$ in our potential cycle both belong to a single edge when $\sigma(v_i)\ne\sigma(v_{i+1})$, the fourth is the probability that the vertex pair belongs to a non-monochromatic edge when $\sigma(v_i)=\sigma(v_{i+1})$.
		
Next, we let $T_\ell$ be the {set} of all possible types of types of length $\ell$ such that $t_{i+1}\ne t_i$ and $t_1\ne t_\ell$. Further, let $T^{(s)}_\ell$ be the {set of types} where the conditions  $t_{i+1}\ne t_i$ or $t_1\ne t_\ell$ fail precisely $s$ times. Clearly 
$\vert T^{(s)}_\ell\vert=\binom{\ell}{s}\vert T_{\ell-s}\vert$ and we know from \cite[Claim 4.2]{AminSilent} that $\vert T_\ell\vert=(\cols-1)^\ell+(-1)^\ell(\cols-1)$. {As (\ref{Cycle_typeest}) does not depend on the particular type $t\in T_\ell$, it follows that}
\begin{align*}
	\mathbb{E}[D_\ell\mid \mathcal{V}(\sigma)]
   &=\sum_{\substack{t\in T^{(s)}_{\ell-s},\\s\in[\ell]}}\mathbb{E}[D_{\ell,t,s}\mid \mathcal{V}(\sigma)]\\
			&=\left[ \frac{c\unif(\unif-1)}{\cols-\cols^{2-\unif}}\right]^{\ell}\sum_{s=0}^{\ell}
				\binom{\ell}{s} \left[(\cols-1)^{\ell-s}+(-1)^{\ell-s}
				(\cols-1)\right]\cdot(1-\cols^{2-\unif})^s \\
   &=\left[ \frac{c\unif(\unif-1)}
            {\cols-\cols^{2-\unif}}\right]^{\ell}\cdot\left[
		\sum_{s=0}^{\ell}
				\binom{\ell}{s} (\cols-1)^{\ell-s}(1-\cols^{2-\unif})^s
		+\sum_{s=0}^{\ell}
				\binom{\ell}{s} (-1)^{\ell-s}
				(\cols-1)(1-\cols^{2-\unif})^s\right] \\
    &=\left[ \frac{c\unif(\unif-1)}
				{\cols-\cols^{2-\unif}}\right]^{\ell}\cdot\left[
				(\cols-\cols^{2-\unif} )^\ell
				+(\cols-1)({-}\cols^{2-\unif})^\ell\right]
			=[c\unif(\unif-1)]^\ell\left[1+\frac{\cols-1}{{(1-\cols^{\unif-1})^\ell}} \right].
\end{align*}
If we divide by $2\ell$ to account for the fact that our cycles above are directed then it follows that
\begin{align*}
\mathbb{E}[C_{\ell,n}\mid \mathcal{V}(\sigma)]\sim \frac{[c\unif(\unif-1)]^\ell}{2\ell}\cdot\left[1+\frac{\cols-1}{{(1-\cols^{\unif-1})^\ell}} \right]{=\lambda_\ell(1+\delta_\ell)=\mu_\ell}.
 \end{align*} 
Since the existence of vertex-disjoint cycles are nearly independent and $\ell, m_\ell$ remain constant as $n$ grows we know that $\mathbb{E}[(C_{\ell,n})_{m_\ell}\mid \mathcal{V}(\sigma)]\sim \mu_\ell^{m_\ell}$. The proposition follows from a standard generalisation of this argument.
\end{proof}

\begin{proposition}\label{SSC_PCycle1}
We have $\mathbb{E}[Y'\mid \mathcal{V}(\sigma)]=O(n^{-1})$.
\end{proposition}

\begin{proof}
Take an arbitrary set $L$ of $\ell$ vertices and let $X$ be the number of edges contained entirely within this set of vertices. For there to be two intersecting cycles in $L$ it must be true that $\ell\leq (k-1)X-1$. Clearly $X\sim $ \Bin$(cn,\binom{\ell}{\unif}/\binom{n}{\unif})$. It follows from a standard application of the Chernoff inequality that
\begin{align*}
	\mathbb{P}\left[X\geq \frac{\ell+1}{\unif-1} \right]\leq \exp\left\{ -\frac{\ell+1}{\unif-1}\ln\left(\frac{\ell+1}{\mathbb{E}[X](\unif-1)}\right)+\frac{\ell+1}{\unif-1}-\mathbb{E}[X]\right\}.
\end{align*}
Let $p_X$ be the right hand side of the above expression.  Then
\[ p_X =O(1)\cdot\exp\left\{\frac{\ell+1}{\unif-1}\ln\mathbb{E}[X]\right\}.\]
Hence $Y'$ is stochastically dominated by a \Bin$(\binom{n}{\ell},p_X)$ random variable. It follows that
\begin{align*}
	\mathbb{E}[Y'] &\leq O(1)\cdot \binom{n}{\ell}\cdot\left(cn\cdot {\binom{\ell}{\unif}}/{\binom{n}{\unif}}\right)^{\frac{\ell+1}{\unif-1}} \\
  &= O(1)\cdot \frac{n^n}{(n-\ell)^{n-\ell}\ell^\ell}\left[cn\cdot \frac{\ell^\ell(n-k)^{n-k}}{n^n(\ell-k)^{\ell-k}}\right]^{\frac{\ell+1}{\unif-1}} \\
  &=O(1)\cdot e^\ell(n-\ell)^\ell\left[cn\cdot e^{-k}(n-k)^{-k}\right]
		= O(n^{-1}),
\end{align*}
completing the proof.
\end{proof}

\medskip

\begin{proof}[Proof of Proposition \ref{Res_CycleImpact}]
Let $x_2,\dots,x_L$ be non-negative integers.  It follows from Lemma \ref{SSC2_L1} that 
\begin{align*}
\mathbb{E}[Z_{\cols,\omega}({\chstar})\mid S]&=\frac{1}{\mathbb{P}[S]}\,
   \sum_{\tau\in\mathcal{B}_{n, \cols}(\omega)}\mathbb{P}[\mathcal{V}(\tau)]\, \mathbb{P}[S\mid \mathcal{V}(\tau)] \\
  &\sim \frac{\prod^L_{\ell=2} \frac{\exp(-\mu_\ell)}{x_\ell!}\mu_\ell^{x_\ell}}{\mathbb{P}
			[S]}\sum_{\tau\in\mathcal{B}_{n,\cols}(\omega)}\mathbb{P}[\mathcal{V}(\tau)] \\
 &= \frac{\prod^L_{\ell=2} \frac{\exp(-\mu_\ell)}{x_\ell!}\mu_\ell^{x_\ell}}{\mathbb{P}
			[S]}\mathbb{E}[Z_{\cols,\omega}({\chstar})].
\end{align*}
Given that $S$ has limiting distribution 
$\prod_{\ell=2}^L\left({\rm Po}(\lambda_\ell)=x_\ell\right)$, we may observe that
\begin{align*}
\frac{\prod^L_{\ell=2} \frac{\exp(-\mu_\ell)}{x_\ell!}\mu_\ell^{x_\ell}}{\mathbb{P}[S]}
		\sim\frac{\prod^L_{\ell=2} \frac{\exp(-\lambda_\ell(1+\delta_\ell))}	
			{x_\ell!}\lambda_\ell^{x_\ell}(1+\delta_\ell)^{x_\ell}}{\prod^L_{\ell=2} 
			\frac{\exp(-\lambda_\ell)}{x_\ell!}\lambda_\ell^{x_\ell}}
		=\prod^L_{\ell=2} (1+\delta_\ell)^{x_\ell}\exp(-\lambda_\ell\delta_\ell)).
\end{align*}
Finally, it is trivial to observe that $\sum_{\ell=2}^{\infty}\lambda_\ell\delta_\ell^2<\infty$.
\end{proof}

\subsection{The second moment}\label{SubSect_TSM}
For two balanced partitions $\sigma,\tau:[n]\rightarrow[\cols]$ we define the overlap $\rho(\sigma,\tau)=(\rho_{ij}(\sigma,\tau))_{i,j\in[\cols]}$ to be the $\cols \times \cols$ matrix with entries $\rho_{ij}(\sigma,\tau)=n^{-1}\vert\sigma^{-1}(i)\cap \tau^{-1}(j)\vert$. Moreover, we introduce the following notation
\begin{align*}
	\rho_{i\star}=\sum_{j\in[\cols]}\rho_{ij},\hspace{1cm}
	\rho_{\bullet\star}=(\rho_{i\star})_{i\in[\cols]},\hspace{1cm}
	\rho_{\star j}=\sum_{i\in[\cols]}\rho_{ij},\hspace{1cm}
	\rho_{\star \bullet}=(\rho_{\star j})_{j\in[\cols]}.
\end{align*}
Define $\bar\rho$ to be the $\cols\times\cols$-matrix with all entries equal to $\cols^{-2}$. 
Further, we let $\eta>0$ be a fixed postive number. Following \cite{AminSilent}, we define
\begin{align*}
	\mathcal{R}_{n,\cols}&=\{\rho(\sigma,\tau):\sigma,\tau:[n]\rightarrow[\cols]\},\hspace{1cm}
	\mathcal{R}^\text{int}_{n,\cols}=\{\rho\in \mathcal{R}_{n,\cols} :\rho_{ij}>1/\cols^3 
		\, \forall\, i,j\in[\cols]\},\\
		\mathcal{R}^\text{bal}_{n,\cols}(\omega)&=\{\rho\in \mathcal{R}^{\text{int}}_{n,\cols} 
			:\vert\rho_{i\star}-\cols^{-1}\vert\leq \omega^{-1}n^{-1/2},
			\vert\rho_{\star j}-\cols^{-1}\vert\leq \omega^{-1}n^{-1/2}, \, \forall\, i,j\in[\cols]\},\\
		\mathcal{R}^\text{bal}_{n,\cols}(\omega,\eta)&=\{\rho\in \mathcal{R}^{\text{bal}}_{n,\cols}(\omega) :\norm{\rho-\bar\rho}_2\leq \eta\},\hspace{0.55cm}
		\overline{\mathcal{R}}_{\cols}= \overline{\bigcup_{n}\mathcal{R}_{n,\cols}}.
\end{align*}
Here $\overline{\mathcal{R}}_{\cols}\subseteq \mathbb{R}^{q\times q}$ is the algebraic closure
of the union (over $n$) of the sets $\mathcal{R}_{n,q}$.
Let $Z_{\cols,\rho}^{(2)}(\hypdash)$ be the number of pairs $(\sigma,\tau)$ with overlap $\rho$ on $\hypdash$. Linearity of expectation means that \begin{align*}\mathbb{E}\left[\left(Z_{\cols,\omega}({\chstar})\right)^2\right]=\sum_{\rho\in\mathcal{R}^\text{bal}_{n,\cols}(\omega)}\mathbb{E}[Z_{\cols,\rho}^{(2)}({\chstar})].\end{align*} The approach will be to show that this summation is dominated by those $\rho$ that are ``close'' to $\bar{\rho}$. To this end, we set $Z_{\cols,\omega,\eta}^{(2)}(\hypdash)=\sum_{\rho\in\mathcal{R}^\text{bal}_{n,\cols}(\omega,\eta)}Z_{\cols,\rho}^{(2)}(\hypdash)$. The following proposition will be proved later in this section.

\begin{proposition}\label{SM_PEqual}
For $c<(\cols^{\unif-1}-1)\ln \cols$ we have $\mathbb{E}\big[\left(Z_{\cols,\omega}({\chstar})\right)^2\big]\sim \mathbb{E}\big[Z^{(2)}_{\cols,\omega,n^{-5/12}}({\chstar})\big]$.
\end{proposition}

\begin{proposition}\label{SM_PropHighC} 
Take $\cols\geq\cols_0$. There exists positive $\epsilon_q=o_q(1)$ so that for $(\cols^{\unif-1}-1)\ln \cols<c<c_{\text{cond}}-\epsilon_q$ the following is true. There exists an integer-valued random variable $0\leq \widetilde Z_{\cols,\omega}\leq Z_{\cols,\omega}$ that satisfies $\mathbb{E}\big[\widetilde Z_{\cols,\omega}({\chstar})\big]\sim \mathbb{E}\big[ Z_{\cols,\omega}({\chstar})\big]$ and such that  \begin{align*}\mathbb{E}\big[\big(\widetilde Z_{\cols,\omega}({\chstar})\big)^2\big]\leq(1+o(1))\mathbb{E}\big[Z^{(2)}_{\cols,\omega,n^{-5/12}}({\chstar})\big].\end{align*}
\end{proposition}

\noindent For notational convenience we now define the entropy and energy as
	\begin{align*}
	H(\rho)&=-\sum_{i,j\in[\cols]}\rho_{ij}\ln \rho_{ij},&
	E(\rho)&=E_{c,\cols,\unif}(\rho)=c\ln\brk{1-2\cols^{1-\unif}+\norm \rho_{\unif}^{\unif}},
	\end{align*}
where $\norm \rho_{\unif}
$ is the $\ell_\unif$-norm. Let 
$f(\rho):=H(\rho)+E(\rho).$ Further, note that
\begin{align}\label{feqn}
	f(\bar\rho)=2\ln q+2c\ln(1-\cols^{1-\unif}).
\end{align}

\begin{fact}\label{SM_F1}
Fix $\cols\geq 3$ and $c\in(0,\infty)$.
 \begin{enumerate}
 	\item[\emph{(i)}] Let $\rho\in\mathcal{R}^{\text{int}}_{n,\cols}$. Then 
 		\begin{align*}
			 &\mathbb{E}\big[Z^{(2)}_{\cols,\rho}({\chstar})\big]
			 \sim\frac{\sqrt{2\pi}n^{\frac{1-\cols^2}{2}}}{\prod_{i,j=1}^\cols\sqrt{2\pi\rho_{ij}}}
			 	\exp\Bigg\{{nH(\rho)}+cn\ln \left(1-\norm{\rho_{\bullet\star}}_\unif^\unif-
		\norm{\rho_{\star\bullet}}_\unif^\unif+\norm{\rho}_\unif^\unif\right) \\
  &\hspace{6cm}+\frac{c\unif(\unif-1)}{2}\left(1-
		\frac{1-\norm{\rho_{\bullet\star}}_{\unif-1}^{\unif-1}-
		\norm{\rho_{\star \bullet}}_{\unif-1}^{\unif-1}+\norm{\rho}_{\unif-1}^{\unif-1}}{1-\norm{\rho_{\bullet\star}}_\unif^\unif-
		\norm{\rho_{\star 
		\bullet}}_\unif^\unif+\norm{\rho}_\unif^\unif}
		\right)\Bigg\}.
 		\end{align*}
	\item[\emph{(ii)}]  Let $\rho\in \mathcal{R}^\text{bal}_{n,\cols}(\omega)$. Then
 		\begin{align*}
 		&\mathbb{E}\big[Z^{(2)}_{\cols,\rho}({\chstar})\big]
			 \sim    
			 \frac{\sqrt{2\pi}n^{\frac{1-\cols^2}{2}}}{\prod_{i,j=1}^\cols\sqrt{2\pi\rho_{ij}}}
			  \exp\Bigg\{{nf(\rho)}+
		\frac{c\unif(\unif-1)}{2}\left(1-
		\frac{1-2\cols^{2-\unif}+\norm{\rho}_{\unif-1}^{\unif-1}}{1-2\cols^{1-\unif}+\norm{\rho}_\unif^\unif}
		\right)+o(1)\Bigg\}.
 		\end{align*}
\end{enumerate}
\end{fact}

\begin{proof}
First, note that for $\gamma \in (0,1)$ we have 
\begin{align*}
	{\frac{\binom{\gamma n}{\unif}}{\binom{n}{\unif}}} &=\frac{\prod_{t=0}^{k-1} (\gamma n-t)}{\prod_{t=0}^{k-1}  (n-t)}=\gamma\prod_{t=1}^{k-1}\left(1+\frac{\gamma-1}{1-t/n}\right)
	=\gamma\prod_{t=1}^{k-1}\left(1+(\gamma-1)(1+t/n+o(n^{-1}))\right) \\
 &=\gamma\prod_{t=1}^{k-1}\left(\gamma +\frac{t(\gamma-1)}{n}+o(n^{-1})\right)
	=\gamma^\unif+\frac{\gamma^{\unif-1}(\gamma-1)}{n}\cdot \frac{k(k-1)}{2}+o(n^{-1}).
\end{align*}
Therefore the probability that a randomly chosen edge is not monochromatic equals
\begin{align*}
 	&\frac{\binom{n}{\unif}-\sum_{i\in[\cols]}\binom{\rho_{i\star}n}{\unif}-\sum_{j=1}^\cols
\binom{\rho_{j\star}n}{\unif}+\sum_{i,j=1}^\cols \binom{\rho_{ij}n}{\unif}}{\binom{n}{\unif}} \\
  &\hspace{0.2cm}=\left(1-\norm{\rho_{\bullet\star}}_\unif^\unif-
		\norm{\rho_{\star 
		\bullet}}_\unif^\unif+\norm{\rho}_\unif^\unif\right) \\
&\hspace{3cm}-\frac{\unif(\unif-1)}{2n}\left(	
	\norm{\rho_{\bullet\star}}_\unif^\unif-	\norm{\rho_{\bullet\star}}_{\unif-1}^{\unif-1}+
		\norm{\rho_{\star \bullet}}_\unif^\unif-\norm{\rho_{\star \bullet}}_{\unif-1}^{\unif-1}
		-\norm{\rho}_\unif^\unif+\norm{\rho}_{\unif-1}^{\unif-1}
\right)+{o(n^{-1})}.
\end{align*}
Next, we raise both sides to the power $cn$ and expand the logarithm around 
$1-\norm{\rho_{\bullet\star}}_\unif^\unif- \norm{\rho_{\star 
		\bullet}}_\unif^\unif+\norm{\rho}_\unif^\unif$. This yields
\begin{align*}
\mathbb{P}\left[\sigma,\tau \text{ are }\cols\text{-colourings}\right]
   &=\left(\frac{\binom{n}{
 		\unif}-\sum_{i\in[\cols]}\binom{\rho_{i\star}n}{\unif}-\sum_{j=1}^\cols
 \binom{\rho_{j\star}n}{\unif}+\sum_{i,j=1}^\cols \binom{\rho_{ij}n}{\unif}}{\binom{n}{\unif}}\right)^{cn}
 		\\
   &=
	 \exp\Bigg\{cn\ln \left(1-\norm{\rho_{\bullet\star}}_\unif^\unif-
		\norm{\rho_{\star\bullet}}_\unif^\unif+\norm{\rho}_\unif^\unif\right) \\
		&\hspace*{2cm} {} +\frac{c\unif(\unif-1)}{2}\left(1-
		\frac{1-\norm{\rho_{\bullet\star}}_{\unif-1}^{\unif-1}-
		\norm{\rho_{\star \bullet}}_{\unif-1}^{\unif-1}+\norm{\rho}_{\unif-1}^{\unif-1}}{1-\norm{\rho_{\bullet\star}}_\unif^\unif-
		\norm{\rho_{\star 
		\bullet}}_\unif^\unif+\norm{\rho}_\unif^\unif}
		\right)+o(1)\Bigg\}.
\end{align*}
Observe that the number of overlaps $\rho\in\mathcal{R}^\text{int}_{n,\cols}$ is given by
\begin{align*}
{\binom{n}{\rho_{11} n,\dots,\rho_{\cols\cols} n}\sim\frac{\sqrt{2\pi}n^{\frac{1-\cols^2}{2}}}{\prod_{i,j=1}^\cols\sqrt{2\pi\rho_{ij}}}\exp\left\{nH(\rho)\right\}}.
\end{align*}
The first result follows. For the second, we set $\epsilon_i=\rho_{i\star}-q^{-1}$. As $\rho$ is $(\omega,n)$-balanced and $\sum_{i\in[\cols]}\epsilon_i=0$ we have
	\begin{align*}
		\norm{\rho_{\bullet\star}}_\unif^\unif=\sum_{i\in[\cols]} (q^{-1}+\epsilon_i)^\unif=\cols^{1-\unif}+o(n^{-1}).
	\end{align*}
Similarly $\norm{\rho_{\star\bullet}}_\unif^\unif=\cols^{1-\unif}+o(n^{-1})$ and $ \norm{\rho_{\bullet\star}}_{\unif-1}^{\unif-1},\norm{\rho_{\star\bullet}}_{\unif-1}^{\unif-1}=\cols^{2-\unif}+o(n^{-1})$. The second result follows.
\end{proof}

\begin{lemma}\label{SM_LExpAsym}
Take $\cols\geq 3$. There exists positive $\epsilon_q=o_1(1)$ such that for $c<c_{\text{cond}}-\epsilon_q$ we have
 \begin{enumerate}
 	\item[\emph{(i)}] If $\rho\in\mathcal{R}^\text{bal}_{n,\cols}(\omega)$ satisfies $\norm{\rho-\bar\rho}_2\leq n^{-5/12}$ then
 	\begin{align*}
 	 		\hspace{-0.85cm}\mathbb{E}\big[Z^{(2)}_{\cols,\rho}({\chstar})\big]
			 \sim    
			(2\pi n)^{\frac{1-\cols^2}{2}}\cols^{\cols^2}
			  \exp\Bigg\{nf(\bar\rho)
		-\frac{nq^2}{2}\left[1-\frac{{c\unif(\unif-1)}
			}{(\cols^{\unif-1}-1)^2 }\right]\norm{\rho-\bar\rho}_2^2 \Bigg\}.
 	\end{align*}
 	\item[\emph{(ii)}] There exists $\eta=\eta(c,\cols,\unif)>0$ and $A=A(c,\cols,\unif)>0$ such that if $\rho\in\mathcal{R}^\text{bal}_{n,\cols}(\omega)$  satisfies $\norm{\rho-\bar\rho}_2\in(n^{-5/12},\eta)$ then
 	\begin{align*}
 	\mathbb{E}\big[Z^{(2)}_{\cols,\rho}({\chstar})\big]\leq \exp\left\{nf(\bar\rho)-An^{1/6}\right\}.
 	\end{align*}
\end{enumerate}
\end{lemma}

\begin{proof}
	Fix $\rho\in\mathcal{R}^\text{bal}_{n,\cols}(\omega)$. If we set $\epsilon=\rho-\bar\rho$ then since the $\ell_3$-norm is dominated by the $\ell_2$-norm, the Taylor expansion of $H(\rho)$ around $\bar\rho$ yields
	\begin{align*}	H(\rho)
	=H(\bar\rho)+\sum_{i,j\in[\cols]}\left({2\ln(\cols)}-1\right)\epsilon_{ij}
		-\frac{1}{2}\sum_{i,j\in[\cols]}{\cols^{2}}\epsilon_{ij}^2+O(\norm{\epsilon}_2^3)=H(\bar\rho)-
		\frac{\cols^2}{2}\sum_{i,j\in[\cols]}\epsilon_{ij}^2+O(\norm{\epsilon}_2^3).
	\end{align*}
Further, if we take the Taylor expansion of $E$ around $\bar\rho$ then
	\begin{align*}
		E(\rho)&=c\ln\left[1-2\cols^{1-\unif}+\sum_{i,j\in[\cols]}(\cols^{-2}{+} \epsilon_{ij})^\unif\right] \\
&=c\ln\left[1-2\cols^{1-\unif}+\cols^{2-2\unif}+
			\cols^{{4}-2\unif}\binom{\unif}{2}\sum_{i,j\in[\cols]}\epsilon_{ij}^2+O(\norm{\epsilon}_2^3)\right] \\
&=c\ln\left[(1-\cols^{1-\unif})^2+
			\cols^{{4}-2\unif}\binom{\unif}{2}\sum_{i,j\in[\cols]}\epsilon_{ij}^2+O(\norm{\epsilon}_2^3)\right] \\
&=E(\bar\rho)+\frac{c\unif(\unif-1)\cols^{2}}{2(\cols^{\unif-1}-1)^2 }
				\sum_{i,j\in[\cols]}\epsilon_{ij}^2+O(\norm{\epsilon}_2^3).
	\end{align*}
Therefore we have
\begin{align*}
	f(\rho) =f(\bar\rho)
		-\frac{q^2}{2}\left[1-\frac{
			c\unif(\unif-1)}{(\cols^{\unif-1}-1)^2 }\right]\norm{\rho-\bar\rho}_2^2
		+O(\norm{\epsilon}_2^3).
\end{align*}
Since $f$ is smooth around $\bar\rho$, there exists $\eta>0$ and $A>0$ such that for $\norm{\rho-\bar\rho}_2\leq\eta$ we have 
	\begin{align*}
	f(\rho)\leq f(\bar\rho)-A\norm{\rho-\bar\rho}^2_2.
	\end{align*}
	 Now~(ii)  follows from  Fact \ref{SM_F1} after noting that $\norm{\rho}_\unif^\unif\sim \cols^{2-2\unif}+o(1)$ and $\norm{\rho}_{\unif-1}^{\unif-1}\sim \cols^{4-2\unif}+o(1)$.  For~(i), since $\norm{\rho-\bar\rho}_2\leq n^{-5/12}$ we have
	\begin{align*}
	f(\rho) =f(\bar\rho)
		-\frac{q^2}{2}\left[1-\frac{
			c\unif(\unif-1)}{(\cols^{\unif-1}-1)^2 }\right]\norm{\rho-\bar\rho}_2^2
		+O(n^{-5/4}).
	\end{align*}
Finally, application of Fact \ref{SM_F1} yields the required estimate.
\end{proof}

\noindent
We know from Fact \ref{SM_F1}(ii) that the following holds.
\begin{fact}\label{SM_FactF}
	Let $\cols\geq 3$, $c\in(0,\infty)$ and $\rho\in\mathcal{R}^\text{{bal}}_{n,\cols}(\omega)$. Then $\mathbb{E}[Z_{\cols,\rho}^{(2)}({\chstar})]=\exp\{nf(\rho)+O(\ln 
	n)\}.$
\end{fact}
\begin{proof}[Proof of Proposition \ref{SM_PEqual}]
\noindent 
Take $c<(\cols^{\unif-1}-1)\ln \cols$. We know that for fixed $\eta>0$
\begin{align}\label{ExpBroken}
\mathbb{E}\big[\left(Z_{\cols,\omega}({\chstar})\right)^2\big]=
\sum_{\substack{\rho\in\mathcal{R}_{n,q}^\text{bal}(\omega)\\\norm{\rho-\bar{\rho}}_2\geq \eta}}\mathbb{E}\big[Z^{(2)}_{q,\rho}(\hypdash)\big]+
\sum_{\substack{\rho\in\mathcal{R}_{n,q}^\text{bal}(\omega)\\\norm{\rho-\bar{\rho}}_2\in(n^{-5/12},\eta)}}\mathbb{E}\big[Z^{(2)}_{q,\rho}(\hypdash)\big]+
\sum_{\substack{\rho\in\mathcal{R}_{n,q}^\text{bal}(\omega)\\ \norm{\rho-\bar{\rho}}_2\leq n^{-5/12}}}\mathbb{E}\big[Z^{(2)}_{q,\rho}(\hypdash)\big]
\end{align}
and 
\begin{align}\label{SM_PEqual_e2}
\sum_{\substack{\rho\in\mathcal{R}_{n,q}^\text{bal}(\omega)\\\norm{\rho-\bar{\rho}}_2\leq n^{-5/12}}}\mathbb{E}\big[Z^{(2)}_{q,\rho}(\hypdash)\big]=\mathbb{E}[Z_{\cols,\omega,n^{-5/12}}^{(2)}({\chstar})]\geq 
\exp\{nf(\bar\rho)
+O(\ln n)\}.
\end{align}
Our objective is to show that the first two summations in (\ref{ExpBroken}) are insignificant relative to the third. For the first term we know from \cite[$(46)-(52)$]{CathHyp} that for $\rho$ such that $\rho_{i\star}=\rho_{\star i}=q^{-1}$ for all  $i\in[q]$ we have
\begin{align}\label{CathCalc}
f(\rho)&\leq  f(\bar \rho) -\ln\left(1+
\frac{q^{2k-2}\norm{\rho}^k_k-1}{(q^{k-1}-1)^2}\right)\Big((q^{k-1}-1)\ln q-c\Big).
 	\end{align}
Extending this to include $\rho\in \mathcal{R}_{n,q}^\text{bal}$ introduces an $o(1)$ term  that is not of consequence. Finally, we note that the function $\rho\mapsto\cols^{2\unif-2}\norm{\rho}_\unif^\unif$ is convex and obtains a global minimum of $1$ at $\rho=\bar\rho$. Hence it follows from (\ref{SM_PEqual_e2}), (\ref{CathCalc}) and Fact \ref{SM_FactF} that 
\begin{align*}
\mathbb{E}\big[\left(Z_{\cols,\omega}({\chstar})\right)^2\big]=
\sum_{\substack{\rho\in\mathcal{R}_{n,q}^\text{bal}(\omega)\\ \norm{\rho-\bar{\rho}}_2\in(n^{-5/12},\eta)}}\mathbb{E}\big[Z^{(2)}_{q,\rho}(\hypdash)\big]+
(1+o(1))\sum_{\substack{\rho\in\mathcal{R}_{n,q}^\text{bal}(\omega)\\ \norm{\rho-\bar{\rho}}_2\leq n^{-5/12}}}\mathbb{E}\big[Z^{(2)}_{q,\rho}(\hypdash)\big].		
\end{align*}
Next, we note that  $\vert\mathcal{R}^\text{bal}_{n,\cols}(\omega,\eta)\vert$ is of polynomial size, hence Lemma \ref{SM_LExpAsym}~(ii) yields
	\begin{align}\label{SM_PEqual_e1}
		\sum_{\substack{\rho\in\mathcal{R}^\text{bal}_{n,\cols}(\omega)\\\norm{\rho-\bar\rho}_2\in(n^{-5/12},\eta)}}\mathbb{E}[Z^{(2)}_{\cols,\rho}({\chstar})]
		\leq \exp\left\{nf(\bar\rho)-An^{1/6}+O(\ln  n)\right\}.
	\end{align}
Hence
\begin{align*}
\mathbb{E}\big[\left(Z_{\cols,\omega}({\chstar})\right)^2\big]=
 (1+o(1))\,\sum_{\substack{\rho\in\mathcal{R}_{n,q}^\text{bal}(\omega)\\ \norm{\rho-\bar{\rho}}_2\leq n^{-5/12}}}\mathbb{E}\big[Z^{(2)}_{q,\rho}(\hypdash)\big].
\end{align*}
The result follows.
\end{proof}
Define $\mathcal{B}$ to be the set of $\rho\in\mathcal{R}_{n,q}$ such that $\rho_{ i\star}=\rho_{\star i}\leq {q^{-1}n^{-1/2}}$ for all $i\in[\cols]$.
As in \cite{ACGCol}, we call $\rho(\sigma,\tau)$ \emph{separable} if for all $i,j\in [q]$,
		\begin{align*}
	  \rho_{ij}(\sigma,\tau)\not\in(\cols^{-1}\myconst,\cols^{-1}(1-\kappa)) 
	\end{align*}
where $\kappa=\cols^{1-\unif}\ln^{20}\cols$.
	Additionally, for {$s\in[\cols]$} we say that $\rho\in\mathcal{B}$ is {\em $s$-stable} if there are precisely $s$ pairs $(i,j)$ such that $\rho_{ij}>q^{-1}(1-\kappa)$. Finally, a $q$-colouring $\sigma$ is \emph{separable} if for all other $q$-colourings $\tau$, the matrix $\rho(\sigma,\tau)$ is \emph{separable}. The main technical accomplishment of \cite{ACGCol} is the following:
	\begin{lemma}\label{SM_PropofF}
		\cite[Lemma 5.2]{ACGCol} Take $\cols>\cols_0$. There exists positive $\epsilon_q=o_q(1)$ such that for $c<c_{\text{cond}}-\epsilon_q$ the following statements are true:
		\begin{enumerate}
			\item[\emph{(i)}]  If $1\leq s<\cols$ then for all separable $s$-stable $\rho\in\mathcal{B}$ we have $f(\rho)<f(\bar\rho)$.
			\item[\emph{(ii)}] If $\rho\in\mathcal{B}$ is $0$-stable and $\rho\ne\bar\rho$ then $f(\rho)<f(\bar\rho)$.
			\item[\emph{(iii)}] {If $c=(\cols^{\unif-1}-1/2)\ln\cols -2$ then for all separable, $\cols$-stable $\rho\in\mathcal{B}$ we have $f(\rho)<f(\bar\rho)$.}
		\end{enumerate}
	\end{lemma}
	\noindent
	The third part of Lemma \ref{SM_PropofF} does not appear in full in \cite{ACGCol} and so we include it here for completeness.

	\begin{proof}[Proof of Lemma \ref{SM_PropofF} (iii)]
  Assume that we have some $q$-stable overlap matrix $\rho$. Since $\rho$ is $q$-stable we know that 
for all $i\neq j \in [q]$,
\begin{align*}
  1-\kappa\leq \cols\rho_{ii}\leq 1 \qquad \text{ and } \qquad \cols_{ij} \leq \kappa 
  \end{align*}
where $\kappa=\cols^{1-\unif}\ln^{20}\cols$.
The intention is to show that $f(\rho)\leq f(\bar\rho)$. However, we will instead show that $f(\rho)\leq f(\cols^{-1}\id)\leq  f(\bar\rho)-2\cols^{1-\unif}$.
To this end, set $c=(\cols^{\unif-1}-1/2)\ln q-2$ and note that
  \begin{align*}
  	f(\bar\rho)=2\ln q+2c\ln (1-\cols^{1-\unif})\hspace{1cm}\text{and}\hspace{1cm}
  	f(\cols^{-1}\id)=\ln q +c \ln (1-\cols^{1-\unif}).
  \end{align*}
Therefore
\begin{align}\label{distcentre}
	f(\bar\rho)-f(\cols^{-1}\id)&=\ln q +\left((\cols^{\unif-1}-1/2)\ln q-2\right) \ln (1-\cols^{1-\unif})\nonumber
			\\
&=\ln q-\left((\cols^{\unif-1}-1/2)\ln q-2\right) \left(\cols^{1-\unif}+\cols^{2-2\unif}/2+O_q(q^{3-3k}) \right)\nonumber
			\\
&=\ln q - \ln q+\frac{\ln q}{2q^{1-\unif}}+2\cols^{1-\unif}-\frac{\ln q}{2\cols^{\unif-1}}+O(\cols^{-k})={2\cols^{1-\unif}+O(\cols^{-k})}.
\end{align}
Since $H$ is concave we have
\begin{align*}
	H(\rho)-H(\cols^{-1}\id)
		&=-\ln q-\sum_{i\in[\cols]}\rho_{ii}\ln \rho_{ii}-\sum_{i\ne j\in[\cols]}\rho_{ij}\ln \rho_{ij}
		\\
&\leq-\ln q-\sum_{i\in[\cols]}\rho_{ii}\ln \rho_{ii}-
		\sum_{i\ne j\in[\cols]}\rho_{ij}\ln\Big(\sum_{i\ne j\in[\cols]} \rho_{ij}\Big)
		\\
&=-\ln q+\sum_{i\in[\cols]}\Big[-\rho_{ii}\ln \rho_{ii}-
		(\cols^{-1}-\rho_{ii})\ln\left(\frac{\cols^{-1}- \rho_{ii}}{\cols-1}\right)\Big]
		\\
&=-\ln q+\sum_{i\in[\cols]}\Big[-\rho_{ii}\ln \rho_{ii}-
		(\cols^{-1}-\rho_{ii})\ln\left({\cols^{-1}- \rho_{ii}}\right)
		+(\cols^{-1}-\rho_{ii})\ln\left({\cols-1}\right)\Big],
\end{align*}
and since $E$ is convex we have
\begin{align*}
E(\rho)-E(\cols^{-1}\id)
		&\leq \frac{\partial E}{\partial \norm{\rho}_\unif^\unif}\left(\norm{\rho}_\unif^\unif-\norm{\cols^{-1}\id}^\unif_\unif\right)
		\\
&=\frac{c}{(1-\cols^{1-\unif})}\left(\sum_{i,j\in[\cols]}\rho_{ij}^k-\cols^{1-\unif}\right)
		\\
&=\frac{c}{(1-\cols^{1-\unif})}\sum_{i\in[\cols]}\left(\rho_{ii}^k-\cols^{-\unif}\right)+O(\cols^{-\unif})
		\\
&=\frac{c}{q^{\unif}(1-\cols^{1-\unif})}\sum_{i\in[\cols]}\left((\cols\rho_{ii})^k-1\right)+O(\cols^{-\unif})
		\\
&=\frac{(\cols^{\unif-1}-1/2)\ln q-2}{q^{\unif}(1-\cols^{1-\unif})}\sum_{i\in[\cols]}\left((\cols\rho_{ii})^k-1\right)+o(\cols^{1-\unif})
		\\
&=\frac{\ln \cols}{\cols}\sum_{i\in[\cols]}\left((\cols\rho_{ii})^k-1\right)+o(\cols^{1-\unif}).
\end{align*} 
Finally then,
\begin{align*}
&f(\rho)-f(\cols^{-1}\id)
=-\ln q+o(\cols^{1-\unif}) \\
&\hspace{1cm}+\sum_{i\in[\cols]}\Big[-\rho_{ii}\ln \rho_{ii}-
		(\cols^{-1}-\rho_{ii})\ln\left({\cols^{-1}- \rho_{ii}}\right)
		+(\cols^{-1}-\rho_{ii})\ln\left({\cols-1}\right)+\frac{\ln q}{q}\big((\cols\rho_{ii})^k-1\big)\Big].
\end{align*}
We now concentrate on one summand at a time. We have
\begin{align*}
	g(\rho_{ii})&=-\rho_{ii}\ln \rho_{ii}-
		(\cols^{-1}-\rho_{ii})\ln\left({\cols^{-1}- \rho_{ii}}\right)
		+(\cols^{-1}-\rho_{ii})\ln\left({\cols-1}\right)+\frac{\ln q}{q}\big((\cols\rho_{ii})^k-1\big),
		\end{align*}
and hence
\begin{align*}
	g'(\rho_{ii})=&-\ln\rho_{ii}	+\ln\left({\cols^{-1}- \rho_{ii}}\right)-\ln(\cols-1)+(\cols\rho_{ii})^{\unif-1}\unif\ln\cols	\\=&\ln\left((\cols\rho_{ii})^{-1}-1\right)-\ln(\cols-1)+(\cols\rho_{ii})^{\unif-1}\unif\ln\cols.
\end{align*}
Next we set $z=(q\rho_{ii})^{-1}$ and observe that $g'(\rho_{ii})=0$ when $z=1+ (q-1)/q^{kz^{1-k}}$.
We can now use this formula to develop an approximation to the solution. Taking $z_0=1$ we have $z_1=1+\frac{\cols-1}{\cols^{\unif}}$ and
\begin{align*}
z_2&=1+\frac{\cols-1}{\cols^{k\big( 1+\frac{\cols-1}{\cols^{\unif}}\big)}}
=1+\frac{\cols-1}{\cols^\unif }\cdot e^{-\frac{k(q-1)}{q^k}\ln q}
=1+q^{1-k}-q^{-k}+O_q(q^{2-2k}\ln q).
\end{align*}
Finally then, $\rho_{ii}=q^{-1}-q^{-k}+q^{-k-1}+O(q^{1-2k}\ln q)$ and so
\begin{align}\label{distcentre2}
&  f(\rho)-f(\cols^{-1}\id)\nonumber \\
&\leq-\ln q+o(\cols^{1-\unif})\nonumber \\
&\hspace{1cm}+q\Big[-(q^{-1}-q^{-k}\ln \rho_{ii}-
		(\cols^{-1}-\rho_{ii})\ln\left({\cols^{-1}- \rho_{ii}}\right)
		+(\cols^{-1}-\rho_{ii})\ln\left({\cols-1}\right)+\frac{\ln q}{q}\big((\cols\rho_{ii})^k-1\big)\Big]\nonumber
	\\
&=-\ln q+o(\cols^{1-\unif})\nonumber\\
&\hspace{1cm}+(1-q^{1-k}+q^{-k})(\ln q+q^{1-k}-q^{-k})+(q^{1-k}-q^{-k})(k+1)\ln q -k\ln q(q^{1-k}+q^{-k})\nonumber \\
&=q^{1-k}+o(\cols^{1-\unif}).
\end{align}
The result follows readily from combining (\ref{distcentre}) and  (\ref{distcentre2}). 
	\end{proof}
	
\medskip
	
We define the {\em cluster} of a $\cols$-coloring $\sigma$ of a fixed hypergraph {$H'$} as the set
	\begin{align*}
	\mathcal{C}({H'},\sigma) &=\cbc{\tau\in\mathcal{B}\, :\, \tau \text{ a $q$-colouring of $H'$, } \,\,\,
   \min_{i\in\brk\cols}{\rho_{ii}(\sigma,\tau)}>\cols^{-1}\myconst}.
	\end{align*}
 The next lemma follows from the proof of \cite[Lemma 4.4]{ACGCol} .
\begin{lemma}\label{SM_VanishingSep}
There exists positive $\epsilon_q=o_q(1)$ and positive integer $q_0$ so that for $q > q_0$ the following holds:
\begin{enumerate}
	\item[\emph{(i)}] \emph{\cite[Lemma 4.3]{ACGCol}}\, 
For $c<c_\text{cond}-\epsilon_q$ the expected number of $(\omega,n)$-balanced $\cols$-colourings of ${\hypdash}$ that are not separable is of size $o(\mathbb{E}[Z_{\cols,\omega}({\chstar})])$.
	\item[\emph{(ii)}] \emph{\cite[Lemma 4.4]{ACGCol}}\, For $(\cols^{\unif-1}-1/2)\ln\cols -2\leq c<c_\text{cond}-\epsilon_q$ the expected number of $(\omega,n)$-balanced $\cols$-colourings $\sigma$ {of $\hypdash$} such that 
$|\mathcal{C}({{\hypdash}},\sigma)|>\mathbb{E}[Z_{\cols,\omega}({\chstar})]/n
	    $ is of size $o(\mathbb{E}[Z_{\cols,\omega}({\chstar})])$.
\end{enumerate}
\end{lemma}

\begin{proof}[Proof of Proposition \ref{SM_PropHighC}]
We have the necessary results to prove Proposition \ref{SM_PropHighC}. To this end we consider two cases:
\begin{enumerate}
	\item[\textbf{Case 1:} \hspace{-1cm}]\hspace{1cm}$c\leq (\cols^{\unif-1}-1/2)\ln\cols -2$. Let $\widetilde{\mathcal{Z}}_{\cols,\omega}$ be the number of separable $(\omega,n)$-balanced $\cols$-colourings of ${\chstar}$. Then by the first part of Lemma \ref{SM_VanishingSep} we know that $\mathbb{E}[\widetilde{\mathcal{Z}}_{\cols,\omega}({\chstar})]\sim\mathbb{E}[{{Z}}_{\cols,\omega}({\chstar})]$. Furthermore, if $c=(\cols^{\unif-1}-1/2)\ln\cols -2$ then by Lemma \ref{SM_PropofF} we know that  $f(\rho)<f(\bar\rho)$ for any separable $\rho\in\mathcal{B}\backslash\{\bar\rho\}$. 
We note that $f(\rho)$ is the sum of the concave function $H(\rho)$ and the convex function $E(\rho)$ which attain their respective maximum and minimum at $\bar\rho$. Further, since $H$ is independent of $c$ and $E$ is a linear multiple of $c$, it follows that reducing  the value of $c$ makes the minimum of $E$ at $\bar\rho$ more shallow and the maximum of $f(\rho)$ more pronounced. Therefore the result holds for $c<(\cols^{\unif-1}-1/2)\ln\cols -2$. 
By Lemma~\ref{SM_PropofF}~(i),~(iii) and Fact~\ref{SM_FactF} we know that
	\begin{align}\label{Case1eqn1}
		\mathbb{E}\big[\big(\widetilde Z_{\cols,\omega}({\chstar})\big)^2\big]=\sum_{\rho\in\mathcal{R}^\text{bal}_{n,\cols}(\omega)}\mathbb{E}[Z_{\cols,\rho}^{(2)}({\chstar})]\leq
  (1+o(1))\, \sum_{\substack{\rho\in\mathcal{R}^\text{bal}_{n,\cols}(\omega)\\ \rho\text{ is }0\text{-stable}}}\mathbb{E}[Z_{\cols,\rho}^{(2)}({\chstar})].
			\end{align}
	Further, it follows from Lemma~\ref{SM_PropofF}(ii)  that
				\begin{align}
	\sum_{\substack{\rho\in\mathcal{R}^\text{bal}_{n,\cols}(\omega)\\ \rho\text{ is }0\text{-stable}}}\mathbb{E}[Z_{\cols,\rho}^{(2)}({\chstar})]\leq(1+o(1)) \mathbb{E}\big[Z^{(2)}_{\cols,\omega,\eta}({\chstar})\big].
	\end{align}Finally, we have by Lemma \ref{SM_LExpAsym}(ii) that
	 \begin{align}\label{Case1eqn2}\mathbb{E}\big[Z^{(2)}_{\cols,\omega,\eta}({\chstar})\big]=(1+o(1)) \mathbb{E}\big[Z^{(2)}_{\cols,\omega,n^{-5/12}}({\chstar})\big].\end{align}
Combining $(\ref{Case1eqn1})-(\ref{Case1eqn2})$ yields
\begin{align*}
	\mathbb{E}\big[\big(\widetilde Z_{\cols,\omega}({\chstar})\big)^2\big]\leq (1+o(1)) \mathbb{E}\big[Z^{(2)}_{\cols,\omega,n^{-5/12}}({\chstar})\big],
\end{align*}
	as claimed.
	\item[\textbf{Case 2:} \hspace{-1cm}]\hspace{1cm}$(\cols^{\unif-1}-1/2)\ln\cols -2<c<c_\text{cond}-\epsilon_q$. Let $\widetilde{\mathcal{Z}}_{\cols,\omega}$ be the number of separable $(\omega,n)$-balanced  $\cols$-colourings $\sigma$ of ${\chstar}$ such that $|\mathcal{C}({\chstar},\sigma)|\leq \mathbb{E}[Z_{\cols,\omega}(\hyp)]/n$. Lemma \ref{SM_VanishingSep}(i) tells us that $\mathbb{E}[\widetilde{\mathcal{Z}}_{\cols,\omega}({\chstar})]\sim\mathbb{E}[{{Z}}_{\cols,\omega}({\chstar})]$.  Moreover, by Lemma \ref{SM_PropofF}(i) and Fact \ref{SM_FactF} we know that
		\begin{align}\label{case2eqn1}{
		\hspace{-2mm}\mathbb{E}\big[\big(\widetilde Z_{\cols,\omega}({\chstar})\big)^2\big]=}
&\sum_{\rho\in\mathcal{R}^\text{bal}_{n,\cols}(\omega)}\mathbb{E}[\widetilde Z_{\cols,\rho}^{(2)}({\chstar})]
	\nonumber\\
&\leq(1+o(1))\left[
		\sum_{\substack{\rho\in\mathcal{R}^\text{bal}_{n,\cols}(\omega)\\ \rho\text{ is }0\text{-stable}}}\mathbb{E}[\widetilde Z_{\cols,\rho}^{(2)}({\chstar})]
		+
		\sum_{\substack{\rho\in\mathcal{R}^\text{bal}_{n,\cols}(\omega)\\ \rho\text{ is }\cols\text{-stable}}}\mathbb{E}[\widetilde Z_{\cols,\rho}^{(2)}({\chstar})]
		\right].
	\end{align}
Let $\tau$ be a $(\omega,n)$-balanced $q$-colouring. Apply Lemma \ref{SM_VanishingSep}(ii) and recall that $|\mathcal{C}({{\hypdash}},\tau)|\leq Z_{\cols,\omega}({\chstar})$, then
\begin{align*}
\mathbb{E}\left[|\mathcal{C}({{\hypdash}},\tau)|\right]\leq \mathbb{E}[Z_{\cols,\omega}({\chstar})]\cdot \frac{o(\mathbb{E}[Z_{\cols,\omega}({\chstar})])}{\mathbb{E}[Z_{\cols,\omega}({\chstar})]}+
\frac{1}{n}\cdot\mathbb{E}[Z_{\cols,\omega}({\chstar})]=o\left(\mathbb{E}[Z_{\cols,\omega}({\chstar})]\right),
\end{align*}
where the first term handles the case $|\mathcal{C}({{\hypdash}},\tau)|>\mathbb{E}[Z_{\cols,\omega}({\chstar})]/n$ and the second  handles the case where $|\mathcal{C}({{\hypdash}},\tau)|\leq\mathbb{E}[Z_{\cols,\omega}({\chstar})]/n$. Further, adapting the proof of \cite[Lemma 5.4]{ACGCol} yields
\begin{align}\label{case2eqn2}
\sum_{\substack{\rho\in\mathcal{R}^\text{bal}_{n,\cols}(\omega)\\\rho\text{ is }\cols\text{-stable}}}\mathbb{E}[\widetilde Z_{\cols,\rho}^{(2)}({\chstar})]
\leq q!\cdot\mathbb{E}[{{Z}}_{\cols,\omega}({\chstar})]\cdot o\left(\mathbb{E}[{{Z}}_{\cols,\omega}({\chstar})]\right)=o\left(\mathbb{E}[{{Z}}_{\cols,\omega}({\chstar})]^2\right).
\end{align}
As previously, it follows from Lemma \ref{SM_PropofF}(ii), Fact \ref{SM_FactF} and Lemma \ref{SM_LExpAsym} that 
\begin{align}\label{case2eqn3}
\sum_{\substack{\rho\in\mathcal{R}^\text{bal}_{n,\cols}(\omega)\\\rho\text{ is }0\text{-stable}}}\mathbb{E}[\widetilde Z_{\cols,\rho}^{(2)}({\chstar})]
\leq(1+o(1))
\mathbb{E}\big[\widetilde Z^{(2)}_{\cols,\omega,\eta}({\chstar})\big]=(1+o(1)) \mathbb{E}\big[\widetilde Z^{(2)}_{\cols,\omega,n^{-5/12}}({\chstar})\big].
\end{align}
Finally, if we combine $(\ref{case2eqn1})-(\ref{case2eqn3})$ and recall that $\mathbb{E}\big[\big(\widetilde Z_{\cols,\omega}({\chstar})\big)^2\big]\geq \mathbb{E}[{{Z}}_{\cols,\omega}({\chstar})]^2$ then
\begin{align*}
\mathbb{E}\big[\big(\widetilde Z_{\cols,\omega}({\chstar})\big)^2\big]
\leq(1+o(1))\mathbb{E}\big[\widetilde Z^{(2)}_{\cols,\omega,n^{-5/12}}({\chstar})\big],
		\end{align*}as required.
\end{enumerate}\end{proof}
\noindent  We define
\begin{align*}
	\Gamma(c,\cols,\unif)=\frac{c\unif(\unif-1)}{2}\cdot\frac{(\cols-1)(2\cols^{\unif}-\cols-1)}{(\cols^{\unif-1}-1)^2}
\hspace{0.5cm}\text{and}\hspace{0.5cm}\Psi(c,\cols,\unif)=1-\frac{c\unif(\unif-1)}{(\cols^{\unif-1}-1)^2}.
\end{align*}
The following proposition  follows readily from \cite[Proposition 5.6]{AminSilent}  with appropriate modifications.
 (Proof omitted.) 

\begin{proposition}\label{SM_PFinalCalc}
For $q\geq 3$ there exists positive $\epsilon_q=o_q(1)$ so that the following is true. If $c<c_{\text{cond}}-\epsilon_q$ we have
\begin{align*}
\mathbb{E}[Z_{\cols,\omega,n^{-5/12}}^{(2)}({\chstar})]\sim
			  			(2\pi n)^{1-\cols}
			 \,\cols^{\cols} \vert\mathcal{B}_{n,\cols}(\omega)\vert^2\,\Psi^{-\frac{(\cols-1)^2}{2}}\exp\Bigg\{nf(\bar\rho)
		+\Gamma(c,\cols,\unif)\Bigg\}.
\end{align*}
\end{proposition}

\begin{proof}[Proof of Proposition \ref{Res_SimpleSecond} and Proposition \ref{Res_HardSecond}]
\noindent{ Recall that {for $\ell \geq 2$}},
\begin{align*}
	\lambda_\ell=\frac{(c\unif(\unif-1))^\ell}{2\ell},\hspace{1cm}\delta_\ell=\frac{\cols-1}{(\cols^{\unif-1}-1)^\ell}.
\end{align*}
It follows that
	\begin{align*}
		\exp\left\{\sum_{\ell=2}^\infty\lambda_\ell\delta_\ell^2\right\}
		&=\exp\left\{
  \sum_{\ell=2}^\infty \frac{(c\unif(\unif-1))^\ell}{2\ell}\cdot 
   \left[ \frac{\cols-1}{(\cols^{\unif-1}-1)^\ell}\right]^2\right\} \\
  &=\exp\left\{\frac{(\cols-1)^2}{2}\left[-\ln\left(1-\frac{c\unif(\unif-1)}
			{(\cols^{\unif-1}-1)^2}\right)-\frac{c\unif(\unif-1)}{(\cols^{\unif-1}-1)^2}\right]\right\} \\
  &=\left[1-\frac{c\unif(\unif-1)}{(\cols^{\unif-1}-1)^2}\right]^{-\frac{(\cols-1)^2}{2}}
			\exp\left\{-\frac{c\unif(\unif-1)(\cols-1)^2}{2(\cols^{\unif-1}-1)^2}\right\}.
	\end{align*}
We require estimates of $\mathbb{E}[Z^2_{\cols,\omega}({\chstar})]$ and $\mathbb{E}[Z_{\cols,\omega}({\chstar})]$. 
For $c<(\cols^{\unif-1}-1)\ln \cols$, it follows from Proposition \ref{SM_PEqual} and Proposition \ref{SM_PFinalCalc} that 
\begin{align*}
&\mathbb{E}\big[\left(Z_{\cols,\omega}({\chstar})\right)^2\big]\sim\mathbb{E}\big[Z^{(2)}_{\cols,\omega,n^{-5/12}}({\chstar})\big]\sim(2\pi n)^{1-\cols}
			 \,\cols^{\cols} \vert\mathcal{B}_{n,\cols}(\omega)\vert^2\,\Psi^{-\frac{(\cols-1)^2}{2}}\exp\Bigg\{nf(\bar\rho)
		+\Gamma(c,\cols,\unif)\Bigg\}.
\end{align*}
We also know from Proposition \ref{FM_FinalCalc} that
\begin{align*}
\ \mathbb{E}\big[ Z_{\cols,\omega}({\chstar})\big]
\sim(2\pi n)^{\frac{1-\cols}{2}}\cols^{\cols/2}\vert\mathcal{B}_{n,\cols}(\omega)\vert\exp\left\{n\ln q+cn\ln\left(1-\cols^{1-\unif}\right)+\frac{c\unif(\unif-1)}{2}\left(
				\frac{\cols-1}{\cols^{\unif-1}-1}\right)\right\}.
\end{align*}
Therefore, for $c<(\cols^{\unif-1}-1)\ln \cols$ it follows that
\begin{align*}
\frac{\mathbb{E}[Z^2_{\cols,\omega}({\chstar})]}{\mathbb{E}[Z_{\cols,\omega}({\chstar})]^2}
	&\sim \Psi^{-\frac{(\cols-1)^2}{2}}\exp\left\{\frac{c\unif(\unif-1)}{2}\cdot\frac{(\cols-1)(2\cols^{\unif}-\cols-1)}{(\cols^{\unif-1}-1)^2}
		-
		{c\unif(\unif-1)}\cdot\frac{\cols-1}{\cols^{\unif-1}-1}\right\} \\
   &=\Psi^{-\frac{(\cols-1)^2}{2}}\exp\left\{-\frac{c\unif(\unif-1)(\cols-1)^2}{2(\cols^{\unif-1}-1)^2}\right\}=\exp\left\{\sum_{\ell=2}^\infty\lambda_\ell\delta_\ell^2\right\}.
\end{align*}
Further, there exists positive $\epsilon_q=o_q(1)$ and positive integer $q>q_0$ such that for if $q > q_0$ and  $c<c_{\text{cond}}-\epsilon_q$ we have from Proposition \ref{SM_PropHighC} and Proposition \ref{SM_PFinalCalc} that
\begin{align*}
&\frac{\mathbb{E}\big[\big(\widetilde Z_{\cols,\omega}({\chstar})\big)^2\big]}{1+o(1)}\leq \mathbb{E}\big[Z^{(2)}_{\cols,\omega,n^{-5/12}}({\chstar})\big]\sim(2\pi n)^{1-\cols}
			 \,\cols^{\cols} \vert\mathcal{B}_{n,\cols}(\omega)\vert^2\,\Psi^{-\frac{(\cols-1)^2}{2}}\exp\Bigg\{nf(\bar\rho)
		+\Gamma(c,\cols,\unif)\Bigg\}.
\end{align*}
{Finally, Lemma~\ref{SM_VanishingSep}(i) implies that}
\[ \mathbb{E}\big[\widetilde Z_{\cols,\omega}({\chstar})\big]\sim \mathbb{E}\big[ Z_{\cols,\omega}({\chstar})\big].\] 
{Combining this with the above completes the proof of Propostions~\ref{Res_SimpleSecond} and~\ref{Res_HardSecond}.}
\end{proof}


\section{The rigid core}\label{Sect_FrozenCore}

In this section we provide an analysis of the core and ridigity in order to  establish Theorem \ref{IB_Freezing}.

\subsection{Emergence of the core}\label{Emergenceofthecore}
Fix a map $\sigma:[n]\mapsto[q]$ such that there is at least one $k$-uniform hypergraph on $n$ vertices
with $cn$ edges which has $\sigma$ as a $q$-colouring. Let $\hyp'(n,k,cn,\sigma)$ denote a $k$-uniform hypergraph chosen uniformly at random with $cn$ edges chosen with replacement, subject to the condition that no edge is monochromatic under $\sigma$. In this section we will primarily work in the planted model, however, our estimates also apply to $\hyp(n,\unif,cn)$ by Theorem \ref{IB_Contiguity}.

We will say that an edge $e$ is $(v,\gamma)$-{\emph{essential}} if  $\sigma(e\backslash\{v\})=\{\gamma\}$. Let $\hypdash(0)=\hypdash$ and define $\hypdash(i+1)$ to be the hypergraph formed from $\hypdash(i)$ by removing every vertex $v$ that has no $(v,\gamma)$-essential edge for some $\gamma\in[\cols]\backslash \{\sigma(v)\}$. When we remove a vertex we remove all edges incident with it. 
We refer to this process, which creates the sequence of hypergraphs $\hyp'(0)$, $\hyp'(1),\ldots $ as the
\emph{stripping process}. By finiteness, there exists some $j$ such that $\hyp'(j + i) = \hyp'(j)$ for all $i\geq 0$.
We refer to this final hypergraph as the \emph{core} and denote it by $\hyp'_{\text{core}}$.
This definition of the core is similar to that used by Molloy and Restrepo~\cite{MolloyNSAT}.

In order to understand the stripping process, we first need to understand the likelihood of encountering a cycle as we explore the neighbourhood of a vertex.  
Fix a vertex $v$ and let $\mathcal{N}_0(v)=\{v\}$ and $\Lambda_0(v)=\mathcal{N}_0(v)$. For $i\geq 1$, we define:
\begin{align*}
	&\hspace{2cm}{\mathcal{N}_i(v)=\{u\in\hyp:\,\exists\,\text{ edge }e\ni u\text{ with } e\cap\mathcal{N}_{i-1}(v)\ne\emptyset\},} \\
  &{\Lambda_i=\mathcal{N}_i(v)-\mathcal{N}_{i-1}(v),\hspace{1cm}\mathcal{E}_i(v)=\left\{\text{edges }e:e\cap \mathcal{N}_i(v)\ne \emptyset\text{ and } e\cap \mathcal{N}_{i-1}(v)=\emptyset\right\}.}
\end{align*}
Essentially, $\mathcal{N}_i(v)$ is the set of vertices in the depth-$i$ neighbourhood of $v$,
$\Lambda_i(v)$ is the set of vertices added in the $i$th step, and $\mathcal{E}_i(v)$ is the set of edges that 
``protrude'' from the depth-$i$ neighbourhood (see edges $f_1,\dots,f_4$ in Figure \ref{f:thrown} below). Where it causes no confusion, for notational convenience we will often write $\mathcal{N}_i, \Lambda_i, \mathcal{E}_i$ rather than $\mathcal{N}_i(v), \Lambda_i(v), \mathcal{E}_i(v)$. 
\begin{lemma}\label{Core_nhdsize} Fix a positive constant $c$ and let $g(n)$ be a sufficiently slowly-growing function. If $\ell\leq g(n)$ then
	\begin{align*}
	\mathbb{E}\left[\vert\Lambda_{\ell}\vert\right]\leq \left(ck{(k-1)}\right)^{\ell}(1+o(1))
	,\hspace{1cm}
	\mathbb{E}\left[\vert\mathcal{N}_{\ell}\vert\right]\leq \left(ck{(k-1)}\right)^{\ell+1}(1+o(1)).
	\end{align*}
\end{lemma}

\begin{proof}
Conditioned on the size of $\vert\Lambda_i\vert$, the size of $\mathcal{E}_i$ is stochastically dominated by the following random variable
	\begin{align}\label{sizeofE}
	\Bin
	\left(cn,{\vert\Lambda_i\vert{{n}\choose \unif-1}}/{{n\choose\unif}}\right)
	=\Bin
	\left(cn,\frac{\vert\Lambda_i\vert\unif}{n}\left(1+O\left(n^{-1}\right)\right)\right).
	\end{align}
	Therefore, noting that $\vert\Lambda_{i+1}\vert\leq (k-1)\vert\mathcal{E}_i\vert$, we have
	\begin{align}\label{LambdaiEi}
	\mathbb{E}\left[\vert\mathcal{E}_{i}\vert\, \big\vert\,  \vert\Lambda_{i}\vert\right]\leq\vert\Lambda_i\vert ck\left(1+O\left(n^{-1}\right)\right)
	\quad \,\, \text{and}\quad \,\,
\mathbb{E}\left[\vert\Lambda_{i+1}\vert\, \big\vert\,\vert\Lambda_{i}\vert\right]\leq\vert\Lambda_i\vert ck(k-1)\left(1+O\left(n^{-1}\right)\right).
	\end{align} 
Further for any $i\leq \ell$, we have
\begin{align*}
\mathbb{E}[\vert\Lambda_{i}\vert] 
 =\mathbb{E}\left[\mathbb{E}\left[\vert\Lambda_{i}\vert\, \big\vert\, \vert\Lambda_{i-1} \vert\right]\right]
  &\leq ck(k-1)\cdot\mathbb{E}[\vert\Lambda_{i-1}\vert]\left(1+O\left(n^{-1}\right)\right)\\
  &= \left(ck(k-1)\right)^i\left(1+O\left(n^{-1}\right)\right)^{i} \\
 &= \left(ck(k-1)\right)^{i}\left(1+o\left(1\right)\right).
	\end{align*}
Therefore
\begin{align*}
\mathbb{E}[\vert\mathcal{N}_{\ell}\vert]=	\mathbb{E}\left[\sum_{i=0}^{\ell}\vert\Lambda_i\vert\right]=\sum_{i=0}^{\ell}\mathbb{E}\left[\vert\Lambda_i\vert\right]\leq(1+o(1))\sum_{i=0}^{\ell}\left(ck(k-1)\right)^i\leq{\left(ck(k-1)\right)^{\ell+1}},
	\end{align*}
completing the proof.
\end{proof}
\noindent
We now seek to calculate the probability that while exploring the depth-$i$ neighbourhood of the fixed vertex $v$, a cycle is encountered. To this end, let $\xi_i$ be the event 
\begin{align*}
\xi_i=\{\exists\, e\in\mathcal{E}_i: \vert e\cap\mathcal{N}_i\vert\geq 2\}\cup
\{\exists\, e,e'\in\mathcal{E}_i:\vert e\cap e'\vert\geq  1\}.
\end{align*}
When $\xi_i$ occurs it means that a cycle has been created when exploring from the depth-$i$ to depth-$(i+1)$ neighbourhood of $v$. The event $\{\exists\, e\in\mathcal{E}_i: \vert e\cap\mathcal{N}_i\vert\geq 2\}$ occurs when one of the edges exposed in this step contains two depth-$i$ vertices: see Figure \ref{f:thrown}, edge $f_2$. The event $\{\exists\, e,e'\in\mathcal{E}_i:\vert e\cap e'\vert\geq  1\}$ occurs when two edges exposed in this step intersect: see Figure \ref{f:thrown}, edges $f_3$ and $f_4$.

\begin{figure}[ht!]
\begin{center}
\begin{tikzpicture}
\draw[rounded corners] (5.3,4.3) rectangle ++(2.2,0.6);
\draw [fill] (5.6,4.6) circle (0.08);
\draw [fill] (6.4,4.6) circle (0.08);
\draw [fill] (7.0,4.6) circle (0.08);
    \draw [rounded corners,rounded corners=3mm] (5.3,2.7)--(5.3,4)--(6.9,3.35)--cycle;
    \draw [fill] (5.6,3.6) circle (0.08);
    \draw [fill] (5.6,3.1) circle (0.08);
    \draw [fill] (6.3,3.35) circle (0.08);
    \draw[rounded corners,rotate around={-15:(5.275,2)}] (5.275,2) rectangle ++(2.2,0.6);
\draw[rounded corners,rotate around={15:(5.425,1.0)}] (5.425,1.0) rectangle ++(2.2,0.6);
\draw [fill] (5.6,2.2) circle (0.08);
\draw [fill] (5.6,1.375) circle (0.08);
\draw [fill] (6.25,2.075) circle (0.08);
\draw [fill] (6.25,1.5) circle (0.08);
\draw [fill] (7.0,1.8) circle (0.08);

\draw [fill] (0,3) circle (0.1);
\node [left] at (-0.1,3) {$v$};
\draw [fill, fill opacity=0.2] (0,3) -- (5,1) -- (5,5) -- (0,3);
\draw [-] (5,1) -- (6,0.6) -- (6,5.4) -- (5,5);
\node [below] at (2,2) {$\mathcal{N}_{i-1}$};
\draw [decoration={brace, mirror, amplitude=6pt, raise=0.5cm}, decorate] (0,0.7) -- (6,0.7);
\node [below] at (3,-0.1) {$\mathcal{N}_{i}$};
\node [above] at (5.4,5.3) {$\Lambda_{i}$};

\node [right] at (7.6,4.6) {$f_{1}$};
\node [right] at (6.9,3.35) {$f_{2}$};
\node [right] at (7.6,2.05) {$f_{3}$};
\node [right] at (7.6,1.55) {$f_{4}$};


\end{tikzpicture}
\caption{Creation of cycles in the depth-$(i+1)$ neighbourhood}
\label{f:thrown}
\end{center}
\end{figure}
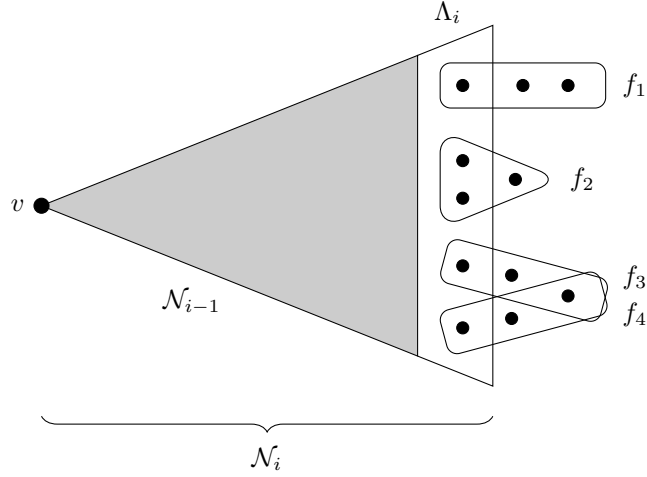

\begin{lemma}\label{cyclesindepthell}
Fix a positive constant $c$ and let $g(n)$ be a sufficiently slowly-growing function. 
If $0\leq i<g(n)$ then there exists $\epsilon>0$ such that 
	\begin{align*}
	{ \mathbb{P}\left[ \bigcup_{j=1}^{g(n)}\xi_j\right]}=O(
	n^{-1+\epsilon}).
	\end{align*}
\end{lemma}

\begin{proof}
	Consider the $\vert\mathcal{E}_i\vert$ edges protruding from $\mathcal{N}_i$. For each edge, ignoring the least-labelled vertex from $\Lambda_i$, we must choose $k-1$ vertices from $\vert V\,\backslash\,\mathcal{N}_{i-1}\vert$ candidates. We model this as a balls into bins argument with $(k-1)\vert\mathcal{E}_i\vert$ balls and $\vert V\,\backslash\,\mathcal{N}_{i-1}\vert$ bins. Since $i\leq  g(n)$, it follows from standard concentration arguments that $\vert\mathcal{N}_{i-1}\vert=o(n)$ and $\vert V\,\backslash\,\mathcal{N}_{i-1}\vert=n(1+o(1))$. 
	
	For $\xi_i$ to occur, either: a ball lands in a certain set of $\vert\Lambda_i\vert$ bins (creating a cycle of the first type), or one of the remaining $\vert  V\backslash\mathcal{N}_{i}\vert $ bins has more than one ball (creating a cycle of the second type). We have
	\begin{align*}
	\mathbb{P}[\xi_i\, \big\vert\, \vert\Lambda_i\vert,\vert\mathcal{E}_i\vert]\leq \sum_{t=1}^{(k-1)\vert\mathcal{E}_i\vert}\frac{\vert\Lambda_i\vert+t}{n}(1+o(1))
	\leq \frac{1}{n} \left[{k\cdot\vert\mathcal{E}_i\vert\cdot\vert\Lambda_i\vert+\tfrac{k^2}{2}\cdot\vert\mathcal{E}_i\vert^2}\right].
	\end{align*}
	For the first term we appeal to the fact that $\mathbb{E}[\vert\mathcal{E}_i\vert\,\big\vert\,\vert\Lambda_i\vert]\leq \vert\Lambda_i\vert ck(1+O(n^{-1}))$. For the second, we note that since $\vert\mathcal{E}_i\vert$ is stochastically dominated by the binomial random variable given in (\ref{sizeofE}), we have 
	\begin{align}\label{varianceresult}
	\mathbb{E}[\vert\mathcal{E}_i\vert^2]\leq (1+o(1))(ck)^2\cdot\mathbb{E}[\vert\Lambda_i\vert^2].
	\end{align} Combining these ideas yields
	\begin{align*}
	\mathbb{P}[\xi_i]\leq\mathbb{E}\left[\mathbb{P}[\xi_i\,\big\vert\,\vert\Lambda_i\vert,\vert\mathcal{E}_i\vert]\right]\leq \frac{c^2k^4}{n}\cdot\mathbb{E} \left[{\vert\Lambda_i\vert^2}\right].
	\end{align*}
Observe that deterministically $\vert\Lambda_{i}\vert\leq (k-1)\vert\mathcal{E}_{i-1}\vert$. Further, utilising this with  (\ref{varianceresult}) and the fact that $\vert\Lambda_0\vert=1$ yields
	\begin{align*}
	\mathbb{P}[\xi_i]\leq \frac{c^2k^4(k-1)^2}{n}\cdot\mathbb{E} \left[{\vert\mathcal{E}_{i-1}\vert^2}\right]\leq \frac{c^4k^8}{n}\cdot\mathbb{E} \left[{\vert\Lambda_{i-1}\vert^2}\right]\leq\dots\leq\frac{(ck^2)^{2(i+1)}}{n}.
	\end{align*}
	Finally, there exists $\epsilon>0$ such that
	\begin{align}\label{cyclesindepthelleqn}
	\mathbb{P}\left[ \cup_{i=1}^{\ell}\xi_i\right]
	\leq \sum_{i=1}^{\ell}\mathbb{P}\left[\xi_i\right]
	\leq \frac{1}{n}\sum_{i=0}^{\ell}(ck^2)^{2(i+1)}=
	n^{-1}{(ck^2)^{2\ell+3}}= O(n^{-1+\epsilon}),
	\end{align}
	completing the proof.
\end{proof}

 \noindent Let $\rho_{i}$ be the probability that a vertex survives $i$ iterations of the stripping process; that is, $\rho_i=\mathbb{P}\left[u\in\hyp'(i)\right]$. The expected number of $(v,\gamma)$-essential edges in $\hypdash(0)$ is given by
 \begin{align}\label{alphavalue}
 \alpha=cn\cdot\frac{{\frac{n}{q}+o(n)\choose {k-1}}}{{n\choose k}-q{\frac{n}{q}+o(n))\choose k}}=\frac{ck}{q^{k-1}-1}+o(1).
 \end{align}

\begin{lemma}
\label{lem:asymptoticallyPoisson}
Fix $v,\gamma$ and let $g(n)$ be an arbitarily slowly growing function. For $i\leq g(n)$, the number of $(v,\gamma)$-essential edges in $\hypdash(i)$ 
has asymptotic distribution 
$\mathrm{Po}(\lambda_i)$ where $\lambda_i=\alpha\rho_i^{k-1}+{o(1)}$.
\end{lemma}

\begin{proof}
Fix $v\in \hyp'(0)$. For any vertex $u\in\hyp'(0)$ consider the event $\zeta_u=\{u\in\hyp'(i)\}$. If $\{\zeta_u\}_{u\in e_{\gamma}\backslash \{v\}}$ are independent for all $(v,\gamma)$-essential edges $e_\gamma$ then $\lambda_i=\alpha\rho_i^{k-1}$. If these events are not independent then it must be that there is a cycle in the depth-$(i+1)$ neighbourhood of $v$. We know by Lemma \ref{cyclesindepthell} that this occurs with probability $n^{-1+\epsilon}$ for some $\epsilon>0$.

 Finally, a straightforward calculation shows that for any $t>0$, the expected number of $t$-tuples of $(v,\gamma)$-essential hyperedges is $\lambda_i^t+{o(1)}$; again, the key point is that if there are no nearby short cycles then the hyperedges occur nearly independently. The method of moments (see \cite[Section 6.1]{PurpleBook}) implies that the number of $(v,\gamma)$-essential edges in $\hypdash(i)$ is distributed as Poisson asymptotically.
  \end{proof}\noindent Fix $v\in\hypdash(i)$. The probability that for all $\gamma\ne\sigma(v)$, there exists a $(v,\gamma)$-essential edge in $\hyp'(i)$  is equal to  \begin{align*}\rho_{i+1}=\left(1-e^{-\lambda_i}\right)^{\cols-1}+{o(1)}.\end{align*} 
Since $\{\rho_i\}_{i\geq 1}$ is positive and non-increasing, we may define $\rho=\lim_{i\rightarrow\infty}\rho_i$. Next, let $\lambda=\lim_{i\rightarrow\infty}\lambda_i=\alpha\rho^{k-1}$. Then
\begin{align}\label{definitionplambdaalpha}
\rho=\left(1-e^{-\lambda}\right)^{\cols-1},\hspace{1cm}\lambda=\alpha\left(1-e^{-\lambda}\right)^{(\cols-1)(\unif-1)},\hspace{1cm}\alpha = \frac{\lambda}{\left(1-e^{-\lambda}\right)^{(\cols-1)(\unif-1)}}.
\end{align}
Define the function
 \begin{align}\label{hsolution}
 h(\lambda)=\frac{\lambda}{\left(1-e^{-\lambda}\right)^{(q-1)(k-1)}}.
 \end{align}
 Let $\alpha_{\rm r}$ be the minimum of $h(\lambda)$ over $\mathbb{R}^+$. 
If $\alpha<\alpha_r$ then there is no solution to $\alpha=h(\lambda)$ and $\mathcal{H}'_{\text{core}}$ is of size $o(n)$. 
 
  \begin{proposition}\label{sizeofcore}
 	{If $\alpha>\alpha_{\rm r}$ then the number of vertices in the core is \whp $ $ given by $\Upsilon(q,k,c)\cdot n+o(n)$ where}
 	\begin{align*}
 	{\Upsilon(q,k,c)=\left[\frac{(\cols^{\unif-1}-1)\cdot \lambda(q,k,c)}{ck}\right]^{\frac{1}{k-1}}}.
 	\end{align*}
 \end{proposition}
 \begin{proof}
By definition of $\rho$, the expected number of vertices in the core is asymptotically equal to $\rho n$.
Recalling $(\ref{alphavalue})$ and the fact that $\lambda=\alpha\rho^{k-1}$, 
\[ \rho n = (\lambda \alpha^{-1})^{\frac{1}{k-1}} = \Upsilon(q,k,c)\cdot n + o(n). \]
The proposition follows by the Chernoff inequality.
\end{proof}
 \begin{lemma}\label{valuelambdar}
The function $h(\lambda)$ defined above has a unique global maximium on $\mathbb{R}^+$ at a value $\lambda_r$ which satisfies
	\begin{align*}
		\lambda_{\rm r}=\ln (q-1)(k-1) + \ln \ln (q-1)(k-1)+o_{\cols,
		\unif}(1).
	\end{align*}
\end{lemma}
\begin{proof}
\noindent
In what follows we set $x=(\cols-1)(\unif-1)$.  Differentiating shows that $h'(\lambda)=0$ if and only if
\begin{align*}
	(1-e^{-\lambda})^x-x\lambda e^{-\lambda}(1-e^{-\lambda})^{x-1}=0
\end{align*}	
	which holds if and only if
	\begin{align}\label{FRED}
	e^{\lambda}-1=x\lambda.
\end{align}
If we substitute $t=-\lambda_{\rm r}-1/x$ then (\ref{FRED}) becomes $te^t=-e^{-1/x}/x$ and the solution is given by 
\begin{align}
\label{solution}
\lambda_{\rm r}=-W_{-1}(-e^{-1/x}/x)-1/x,
\end{align}
where $W_{-1}$ is the non-principal real branch of the Lambert-W function \cite{LambertW}. Applying the recursion  $W_{-1}(s)=\ln(-s)-\ln(-W_{-1}(s))$ yields
 \begin{align*}
 	-W_{-1}(-e^{-1/x}/x)&=-\ln\left(e^{-1/x}/x\right)+\ln(-W_{-1}(-e^{-1/x}/x))
 	\\&=1/x+\ln x+\ln(-W_{-1}(-e^{-1/x}/x)).
 \end{align*}
 Therefore $\lambda_{\rm r}$ is given by
 \begin{align*}
 	\lambda_{\rm r}&=-W(-e^{-1/x}/x)-1/x=\ln x
 		+\ln(-W_{-1}(-e^{-1/x}/x)) \\
   &\hspace{0cm}=\ln x
 		+\ln\left(1/x+\ln x+\ln(-W_{-1}(-e^{-1/x}/x))\right)
 	=\ln x
 		+\ln\left(\ln x \left[1+\frac{1/x+\ln(-W_{-1}(-e^{-1/x}/x))}{\ln x}\right]\right) \\
  &\hspace{0cm}=\ln x+\ln \ln x +\ln\left(1+\frac{x^{-1}+\ln(-W_{-1}(-e^{-1/x}/x))}{\ln x}\right)=\ln x + \ln \ln x+\epsilon(\cols,\unif),
 \end{align*}
where 
\begin{align*}
{e^{\epsilon(\cols,\unif)}>1+\frac{\left[(q-1)(k-1)\right]^{-1}+\ln\ln((q-1)(k-1))}{\ln (q-1)(k-1)}.}
\end{align*}
\end{proof}\noindent
Guided by (\ref{alphavalue}) and (\ref{definitionplambdaalpha}) we define $c_{\rm r}>0$ by 
\begin{align}
\label{exact}
c_{\rm r} =\frac{\cols^{\unif-1}-1}{\unif}\cdot \frac{\lambda_{\rm r}}{(1-e^{\lambda_{\rm r}})^{(q-1)(k-1)}}.
\end{align}
By Lemma \ref{valuelambdar}, if we again set $x=(q-1)(k-1)$ then 
\begin{align*}
c_{\rm r} &
=\frac{\cols^{\unif-1}-1}{\unif}\cdot\frac{\ln x+\ln\ln x+o_{x}(1)}{\left(1-\frac{e^{\epsilon(\cols,\unif)}}{x\ln x}\right)^x} \\
&=\frac{\cols^{\unif-1}-1}{\unif}\left(\ln x+\ln\ln x+o_{x}(1)\right)\left(1+\frac{e^{\epsilon(\cols,\unif)}}{x\ln x}+\left(\frac{e^{\epsilon(\cols,\unif)}}{x\ln x}\right)^2+\cdots\right)^x \\
&=\frac{\cols^{\unif-1}}{\unif}\left(\ln x+\ln\ln x+1+o_{q}(1)\right),
\end{align*}
matching (\ref{valuecr}). It remains to prove that this value of $c_{\rm r}$ marks the rigidity threshold, by proving Theorem \ref{IB_Freezing}.
	
\subsection{Rigidity in the Kempe core}
We define a {\it flippable set} to be a set of vertices $T\subset \hypdash_{\text{core}}$ such that for every $v\in T$ there exists $\gamma\in[q]\,\backslash\,\sigma(v)$ such that for all $(v,\gamma)$-essential edges $e$ we have $(e\backslash\{v\})\cap T\ne \emptyset$. If $\tau$ is any other colouring that differs from $\sigma$ on $\hypdash_{\text{core}}$ then $(\sigma\Delta\tau)\cap \hypdash_{\text{core}}$ is a flippable set. Take $v\in   (\sigma\Delta\tau)\cap \hypdash_{\text{core}}$. As $v$ is in the core we know that there exists an essential edge $e$ such that $\sigma(e\backslash\{v\})=\tau(v)$. Therefore for $e$ to not be monochromatic under $\tau$, an element of $e\backslash\{v\}$ must also be recoloured under $\tau$. 

Let $\Psi_T(v,\gamma)$ be the event that for all $(v,\gamma)$-essential edges $e$ we have $(e\backslash\{v\})\cap T\ne \emptyset$. A flippable set $T$ induces a directed multigraph $D(T)$ on the vertex set of $T$ with arcs defined as follows: for each $v\in T$ and $\gamma\in[q]\backslash \{\sigma(v)\}$ if $\Psi_T(v,\gamma)$ occurs then for every $(v,\gamma)$-essential edge $e$ and vertex $u\in(e\backslash\{v\})\cap T$ we add an arc from $v$ to $u$. Further, let $d_{T}^+(v,\gamma)$ $\left(\text{respectively }d_{T}^-(v)\right)$ be the number of outwardly (respectively inwardly) directed arcs in $D(T)$ that are created from $(v,\gamma)$-essential edges. Directed graphs arise naturally due to the asymmetry present in essential edges.

First we work toward a proof of Theorem \ref{IB_Freezing}(a)(i). A direct first moment calculation on the number of flippable sets is unfortunately not fruitful. This is because a flippable set typically has long paths of vertices in the corresponding directed graph. If the path ends in a vertex with in-degree zero, then we may cut the path at any point and still have a flippable set, thus counting the number of flippable sets of size $t$ misrepresents the situation. When approaching this problem in the bicolouring case, Molloy and Restrepo \cite{MolloyNSAT} defined {\it weakly flippable sets} as  the restriction of a flippable set to a denser part of the hypergraph. This deals with the problem of over counting paths but results in significant technical difficulty (in particular, when you cut the flippable set in this way it is no longer actually a flippable set). 

Instead, armed with contiguity, our approach is to  define what we call {\it $*$-flippable sets} as the largest subset of a flippable set $T$ such that all vertices in the corresponding directed graph have  in-degree greater than or equal to one. Again, if $\tau$ is any other colouring of $\hypdash$ that differs on a vertex in the core then $(\sigma\Delta\tau)\cap \hypdash_{\text{core}}$ contains a $*$-flippable set. 

\begin{lemma}\label{core_inside1}
For $c>c_{\rm r}$ there exists $\xi>0$  such that for all $f(n):\mathbb{N}\mapsto\mathbb{R}$ growing arbitrarily slowly, \whp\ there is no $*$-flippable set $T\subseteq\hypdash_{\text{core}}$ of size $f(n) \leq |T|\leq \xi n$.
\end{lemma}

\begin{proof} 
Take $v\in\hyp_{\text{core}}$. Let $\xi>0$ be an arbitrarily small constant and set $f(n):\mathbb{N}\mapsto\mathbb{R}^+$ to be an arbitrarily slowly growing function. Next we fix  $T\subseteq \hypdash_{\text{core}}$ of size $|T|=t$ where $f(n)\leq t\leq \xi n$. We say that a pair $(v,u)$ is a \emph{candidate} if there exists a $(v,\gamma)$-essential edge $e$ such that $u\in(e\backslash\{v\})\cap T$, irrespective of whether or not $\Psi_T(v,\gamma)$ occurs. Let $\psi_T(v,e)$ be the number of candidates created by a particular $(v,\gamma)$-essential edge $e$. Set $\Delta=\frac{t}{\vert\hyp'_{\text{core}}\vert}$. 
Then for any $i\in \{0,1,\ldots, k-1\}$,
	\begin{align*}
	\mathbb{P}\left[\psi_T(v,e)=i\right]={k-1\choose i}\Delta^s(1-\Delta)^{k-i-1}\,
   (1+o_{\Delta^{-1}}(1))
	\end{align*}
where $o_{\Delta^{-1}}(1)$ is a term that tends to zero as $\Delta$ goes to zero.
Further we set $\Delta_\star=\mathbb{P}\left[\psi_T(v,e)\geq 1\right]$.	Next, let $\phi(v,\gamma)$ be the number of $(v,\gamma)$-essential edges. 
Now $\phi(v,\gamma)$ tends in distribution to $Y\sim {\rm Po}(\lambda)$ 
conditional on the event $\{Y\geq 1\}$, 
by Lemma~\ref{lem:asymptoticallyPoisson} and the fact that 
$v\in\hyp_{\text{core}}$. 
Further, $d_T^+(v,\gamma)>0$ if and only if each $(v,\gamma)$-essential edge creates a candidate. Then for any $s\geq 1$,
	\begin{align*}
	\mathbb{P}\left[(\phi(v,\gamma)=s) \cap\Psi_T(v,\gamma)\right]= 
\left(\frac{\lambda^s/s!}{e^\lambda-1} + o(1)\right)\cdot \Delta_\star^s  
	\end{align*}
	and also
\[
	\mathbb{P}\left[\Psi_T(v,\gamma)\right]=  \frac{e^{\lambda \Delta_\star}-1}{e^{\lambda}-1}\,(1+o_{\Delta^{-1}}(1)).
\]
Therefore
\begin{equation}\label{psigamma}
		\mathbb{P}\left[\bigcup_{\gamma\ne\sigma(v)}\Psi_T(v,\gamma)\right]=  (q-1)\cdot\frac{e^{\lambda \Delta_\star}-1}{e^{\lambda}-1}(1+o_{\Delta^{-1}}(1)).
	\end{equation}
Next, we define 
\[
  \Phi_T(+)=\bigcap_{v\in T}\bigcup_{\gamma\ne\sigma(v)}\Psi_T(v,\gamma).
\]
That is, $\Phi_T(+)$ occurs if and only if for each $v\in T$ there exists $\gamma\ne\sigma(v)$ such that every $(v,\gamma)$-essential edge $e$ satisfies $(e\backslash\{v\})\cap T\ne\emptyset$.  In particular, if $\Phi_T(+)$ holds then $T$ will be a flippable set. 
Note that $\Phi_T(+)$ is the intersection (over $v\in T$) of independent events,
as an edge $e$ can be $(v,\gamma)$-essential for at most one vertex $v$
and colour $\gamma\neq \sigma(v)$, by definition.
Hence by (\ref{psigamma}),
\[  \mathbb{P}\left[\Phi_T(+)\right]=\left((q-1)\cdot\frac{e^{\lambda \Delta_\star}-1}{e^{\lambda}-1}\cdot (1+o_{\Delta^{-1}}(1))\right)^t.
\]
Next, we condition on $\Phi_T(+)$ and calculate the probability that $T$ is  $\ast$-flippable. 

Let $d_T(+)$ be the total number of directed edges in the induced directed graph $D(T)$. Note that since the set of edges which $\{\cup_{\gamma\ne\sigma(v)}\Psi_T(v,\gamma)\}$ and $\{\cup_{\gamma\ne\sigma(v)}\Psi_T(u,\gamma)\}$ depend on for $u\ne v$ are non-overlapping, it follows that the events are independent. Therefore
\begin{align*}
{\mathbb{E}\left[\,d_T(+)\,\vert \,\Phi_T(+)\,\right]}\,&  \leq \frac{t\,(q-1) \left(
	\frac{\lambda\Delta_\star}{e^{\lambda}-1}
	+	\frac{(k-1)\cdot\lambda\cdot\Delta_\star^2/2!}{e^{\lambda}-1}
	+\sum_{s\geq 2}ks\cdot\frac{(\lambda\Delta_\star)^s/s!}{e^{\lambda}-1}\right)}{ (q-1)\cdot\frac{e^{\lambda \Delta_\star}-1}{e^{\lambda}-1}\cdot(1+o_{\Delta^{-1}}(1))} \\
&\hspace{-2cm}{=t\,\left(
	\frac{\lambda\Delta_\star}{e^{\lambda \Delta_\star}-1}
	+	\frac{(k-1)\cdot\lambda\cdot\Delta_\star^2/2!}{e^{\lambda \Delta_\star}-1}
	+\sum_{s\geq 2}ks\cdot\frac{(\lambda\Delta_\star)^s/s!}{e^{\lambda \Delta_\star}-1}\right)
	\cdot(1+o_{\Delta^{-1}}(1))
}{\,\leq (1+o_{\Delta^{-1}}(1))t.}
\end{align*} 
{Fix $\delta>0$ and let $\chi_T(+)$ be the event that $d_T(+)<(1+\delta)t$. Define Bernoulli random variables $X_1,\dots,X_{t(t-1)}$ as follows: for each of the $t(t-1)$ ordered pairs of elements of $T$, ordered lexicographically, if the $i$-th ordered pair is a directed edge in $D(T)$, let $X_i=1$ and let $X_i=0$ otherwise. Since $\{X_i\}_{i\in[t(t-1)]}$ are independent and identically distributed, we know that $\sum_{i=1}^{t(t-1)}X_i$ is distributed binomially. Hence
	\begin{align*}
	\mu:=\mathbb{E}\left[\sum_{i=1}^{t(t-1)}X_i\right]=\frac{t(q-1)(k-1)\lambda\Delta}{e^{\lambda}-1}=g(\lambda)\cdot t\Delta\hspace{0.5cm}\text{where }g(\lambda)\in(0,1)\text{ for }c>c_{\rm r}.
	\end{align*}
	Setting $\gamma=t(1+\delta-g(\lambda)\Delta)$, and noting that $\gamma/\mu=\frac{1+\delta}{g(\lambda)\Delta}-1$, it follows from the Chernoff bound that
	\begin{align*}
		\mathbb{P}[\neg\chi_T(+)]&=\mathbb{P}\left[\sum_{i=1}^{t(t-1)}X_i\geq (1+\delta)t\right]\\&=\mathbb{P}\left[\sum_{i=1}^{t(t-1)}X_i\geq \gamma+\mu\right]
	\\
&	= \exp\left\{-g(\lambda)\cdot t\Delta\left[\left(\frac{1+\delta}{g(\lambda)\Delta}\right)\ln\left(\frac{1+\delta}{g(\lambda)\Delta}\right)-\frac{1+\delta}{g(\lambda)\Delta}+1\right]\right\}
		\\
&	\leq \exp\left\{-t\left[\left(1+\delta\right)\ln\left(\frac{1+\delta}{g(\lambda)\Delta}\right)-(1+\delta)\right]\right\}
			\\&\leq \exp\left\{t\left[\left(1+\delta\right)\ln\left(g(\lambda)\Delta\right)+1\right]\right\}.
	\end{align*}	
	Then, we observe that
	\begin{align*}
		{\vert\hyp_{\text{core}}\vert\choose t}\cdot \mathbb{P}[\neg\chi_T(+)]
		&\leq \exp\left\{t\left[\left(1+\delta\right)\ln\left(g(\lambda)\Delta\right)+1\right]+t\ln\left(e\Delta^{-1}\right)\right\}
		\\
&= \exp\left\{t\left[\ln\left(\frac{(g(\lambda)\Delta)^{1+\delta}}{\Delta}\right)+2\right]\right\}
		\leq \exp\left\{t\left[\ln\left(g(\lambda)^{1+\delta}\Delta^\delta\right)+2\right]\right\}.
	\end{align*}
	Since  $\Delta<\xi/\rho$ and $\xi$ is arbitrary, $\delta,\xi$ may be chosen such that $	{\vert\hyp_{\text{core}}\vert\choose t}\cdot \mathbb{P}[\neg\chi_T(+)]\leq \exp\{-\Omega(t)\}$. 
 Next, we observe that} 
\begin{align*}
3\left[\left(\frac{1+\delta}{2+\delta}\right)^{2+\delta} \left(\frac{1+\delta}{\delta}\right)^{\delta }\right]^t\leq 3^{-t}.
\end{align*}
{Set $\Phi_T(-)$ be the probability that the in-degree of each $v\in T$ is non-zero conditional on the occurrence of $\Phi_T(+)$. If $r=(1+\delta)t$ then the probability that $\Phi_T(-)$ occurs is dominated by a balls-into-bins experiment where we throw $r$ balls into $t$ bins and require that each bin is non-empty. Observe that }
\begin{align*}
{\mathbb{P}\left[\Phi_T(-)\,\vert\,\Phi_T(+)\,\cap\,\chi_T(+)\right]}&\leq{r-1\choose t-1}\Big/ {{r+t-1\choose t-1}} \\
&=\frac{(r-1)!}{(r-t)!}\cdot\frac{r!}{(r+t-1)!} \\
&\sim \frac{(r-1)^{r-1}}{(r+t-1)^{r+t-1}}\cdot \frac{r^r}{(r-t)^{r-t}}\\
&\leq \frac{ 3\cdot[(1+\delta)t]^{2(1+\delta)t}} {[(2+\delta)t]^{(2+\delta)t}[\delta t]^{\delta t} }
\\
&=3\cdot\left[\left(\frac{1+\delta}{2+\delta}\right)^{2+\delta} \left(\frac{1+\delta}{\delta}\right)^{\delta }\right]^t\leq 3^{-t}.
\end{align*}
{Let $F(t)$ be the number of $*$-flippable sets $T\subseteq \hyp'_{\text{core}}$ of size $\vert T\vert=t$. Below we denote by $\Phi_t(+),\Phi_t(-),\chi_t(+)$ the events $\Phi_T(+),\Phi_T(-),\chi_T(+)$ for an arbitrary set $T$ of size $t$. There exists $K(q,k)>0$ such that }
\begin{align}
{	\mathbb{E}[F(t)]}\,\,&{\leq {\vert\hyp'_{\text{core}}\vert\choose t}\cdot\mathbb{P}\left[\Phi_t(+)\right]\cdot\mathbb{P}\left[\Phi_t(-)\,\vert\,\Phi_t(+)\right]\nonumber
}\\
&\leq {\vert\hyp'_{\text{core}}\vert\choose t}\cdot\mathbb{P}\left[\Phi_t(+)\right]\cdot\big(\mathbb{P}\left[\Phi_t(-)\,\vert\,\Phi_t(+)\,\cap\,\chi_t(+)\right]+\mathbb{P}[\neg\chi_t(+)\,\vert\,\Phi_t(+)]\big)]\nonumber
\\
&\leq {\vert\hyp'_{\text{core}}\vert\choose t}\cdot\big(\mathbb{P}\left[\Phi_t(+)\right]\cdot\mathbb{P}\left[\Phi_t(-)\,\vert\,\Phi_t(+)\,\cap\,\chi_t(+)\right]+\mathbb{P}[\neg\chi_t(+)]\big)]\nonumber
\\
&\leq \left[\frac{e}{3}\cdot \Delta^{-1} \cdot(\cols-1)\cdot\frac{e^{\lambda \Delta_\star}-1}{e^{\lambda}-1}\right]^t+\exp\{-\Omega(t)\}]\nonumber \\
&\leq	\left[\frac{(q-1)(k-1)\lambda}{e^{\lambda}-1}\cdot \frac{e}{3}
		\cdot(1+K\Delta)\right]^t+\exp\{-\Omega(t)\}.
\label{factorsmall2}
\end{align}
Further, for $c>c_{\rm r}$ we know from the definition of $\lambda_{\rm r}$ that 
	\begin{align}
	\label{factorsmall}
	\frac{(\cols-1)(\unif-1)\lambda }{e^{\lambda}-1}<1.
\end{align}
 For the remainder of the proof we condition on the  event  $\vert\hyp'_{\text{core}}\vert\geq \tfrac{1}{2}\Upsilon(q,k,c)\cdot n$,  which occurs \whp\ by Proposition \ref{sizeofcore}. Since $\Delta=t/\vert\hyp'_{\text{core}}\vert$ where $t\leq \xi n$,  it follows that $\xi$ may be chosen so that $\frac{e}{3}(1+K\Delta)<1$. Hence from Markov's inequality we have $\mathbb{P}[F(t)>0]\leq\exp\{-\Omega(t)\}$. Finally, there are linearly many values of $t$ between $f(n)$ and $\xi n$, and summing over this range completes the proof.
\end{proof}

\begin{lemma}\label{core_inside2}
For $c>c_{\rm r}$ there exists $\delta>0$ such that for all $g(n):\mathbb{N}\rightarrow\mathbb{R}$ growing arbitrarily slowly, \whp\ there is no flippable set $S\subseteq \hypdash_{\text{core}}$ of size $g(n)\leq |S|\leq \delta n$.
\end{lemma}
\begin{proof}
In this proof we use $\xi$ and $f(n)$ from Lemma \ref{core_inside1}. For a flippable set $S$ define $S^*$ to be the largest $*$-flippable subset of $S$. Let $Y$ be the set of all $*$-flippable sets $B$ such that $\vert B\vert \leq f(n)$ and define 
$X_{(0)}=\cup_{B\in Y}B$. Since the union of two $*$-flippable sets is again $*$-flippable, we know that $X_{(0)}$ is $*$-flippable. 
	
For $\alpha\geq 1$ define
	\begin{align*}
X_{(\alpha)}=X_{(\alpha-1)}\cup\{v\in\hypdash_{\text{core}}\backslash X_{(\alpha-1)}:\exists\text{ }\gamma\text{ such that }\Psi_{X_{(\alpha-1)}}(v,\gamma)\text{ occurs}\}.
\end{align*}
 The process above begins with a $*$-flippable set and iteratively adds vertices that ensure the new set is a flippable set. It is not hard to see that in fact every flippable set can be reconstructed from its $*$-flippable set in this way. Further for $v\in\hypdash_\text{core}\backslash X_{(\alpha)}$, we know that $X_{(\alpha)}\cup\{v\}$ is a flippable set if and only if  $\bigcup_{\gamma\ne\sigma(v)}\Psi_{X_{(\alpha)}}(v,\gamma)$ occurs for some $\gamma\ne\sigma(v)$. We know from (\ref{psigamma}) and (\ref{factorsmall2}) that  there exists $\nu>0$ such the expected size of $|X_{(\alpha)}\backslash X_{(\alpha-1)}|$ is equal to 
\begin{align*}
(1+o(1))\cdot n\rho \cdot\mathbb{P}\left[\bigcup_{\gamma\ne\sigma(v)}\Psi_{X_{(\alpha)}}(v,\gamma)\right]\,\, { \leq } \,\, (1+k\xi)\cdot\frac{(q-1)(k-1)\lambda }{e^{\lambda}-1}\cdot\vert X_{(\alpha)}\vert \leq(1-\nu)\cdot\vert X_{(\alpha)}\vert.
\end{align*} 
To see that this is the case, recall that (\ref{factorsmall}) holds as $c>c_{\rm r}$. Further after multiplying the left hand side of (\ref{factorsmall}) by $(1+k\xi)$ the result is still less than one if $\xi$ is sufficiently small. If we apply this argument inductively, then it follows that $\mathbb{E}[|X_{(\alpha)}\backslash X_{(\alpha-1)}|]\leq(1-\nu)^\alpha\cdot\mathbb{E}[|X_{(0)}|]$. Further, as the process continues we will eventually have $X_{(\alpha+1)}=X_{(\alpha)}=:X_{(\infty)}$ where $\mathbb{E}[\vert X_{(\infty)}\vert]\leq \vert X_{(0)}\vert/\nu$.
	
Finally, if we condition on Lemma \ref{core_inside1} then \whp\ no $*$-flippable set exists of size between $f(n)$ and $\xi n$. It must be that $\vert X_{(0)}\vert\leq f(n)$, otherwise there would be a subset of $Y$ such that the union over this subset would produce a $*$-flippable set of size between $f(n)$ and $\xi n$. Hence, there exists $g(n)$ such that from Markov's inequality  \whp\  we have $\vert X_{(\infty)}\vert \leq (f(n))^2/\nu=g(n)$. For an arbitrary flippable set $A$ we know from Lemma \ref{core_inside1} that either $\vert A^*\vert\leq f(n)$ or $\vert A^*\vert\geq \xi n$. In the second case the lemma follows immediately with $\delta=\xi$. For the first, since $\vert A^*\vert \leq f(n)$ we must have $A_*\subseteq X$, but then $\vert A\vert \leq \vert X_{\infty}\vert\leq g(n)$, as required.
\end{proof}

\noindent
\begin{proof}[Proof of Theorem \ref{IB_Freezing}]
	The proof of  Theorem \ref{IB_Freezing}(a)(i) follows immediately from Lemma~\ref{core_inside1} and Lemma~\ref{core_inside2}.
	
	For Theorem \ref{IB_Freezing}(a)(ii) we must prove that all but a vanishing proportion of vertices outside the core are not $1$-frozen. Fix an integer valued function $g(n)$ which grows arbitrarily slowly and perform $g(n)$ iterations of the stripping process. Take an arbitrary $v_\ast\in\hyp'(0)\backslash\hyp'(g(n))$ such that there are no cycles in the depth $g(n)$-neighbourhood of $v_\ast$. For $j=0,\dots,g(n)$ let $I_j=\left(\hyp'(j)\backslash\hyp'(j+1)\right)\cap\mathcal{N}_{g(n)}(v_{\ast})$ be the set of vertices in the depth-$g(n)$ neighbourhood of $v_\ast$ which are stripped at the $j$-th iteration of the stripping process.  We will show that there exists $t\in\mathbb{N}$ and a sequence of colourings $\sigma=\sigma_0$, $\sigma_1,\dots,\sigma_t$, such that $\vert\sigma_i\Delta\sigma_{i+1}\vert=1$ and $\sigma(v)\ne \sigma_t(v)$ (for convenience, in our sequence of colourings below we will use two indices). By definition, this will show that $v_\ast$ is not $1$-frozen.
	
	For $j=0,\dots,g(n)$ and each vertex $v\in I_j$, there exists some $\gamma\ne\sigma(v)$ such that no $(v,\gamma)$-essential edge exists in $\hyp'(j)$.  Label the vertices in $I_j$ as $v_{j,1},\dots v_{j,\vert I_j\vert}$ (lexicographically) and let $\gamma_{j,i}$ be the smallest colour such that no $(v_{j,i},\gamma_{j,i})$-essential edge exists in $\hyp'(j)$. We define 
	\begin{align}\label{seqcol}
	\sigma_{j,i}(w)=
	\begin{cases}
	\sigma_{j,i-1}(w)&\text{if $w\ne v_{j,i}$},\\
	\gamma_{j,i}&\text{if $w=v_{j,i}$},\\
	\end{cases}
	\end{align}
	where $\sigma_{j,0}=\sigma_{j-1,\vert I_{j-1}\vert}$ and $\sigma_{0,0}=\sigma$. We call this the recolouring process. To see that this process defines a sequence of proper colourings, take an arbitrary $v_{j,i}$ and recall that no $(v_{j,i},\gamma_{j,i})$-essential edge exists in $\hyp'(j)$. This means that every $(v_{j,i},\gamma_{j,i})$-essential edge $e$ has non-empty intersection with $\cup_{s\leq j-1}I_{s}$. However, the colour of all of these vertices has changed during previous steps of the recolouring process, and hence these edges are no longer $(v_{j,i},\gamma_{j,i})$-essential under $\sigma_{j,i}$. Further, since we assume there are no cycles in the depth $g(n)$-neighbourhood of $v_\ast$, no `new' essential edges can be created during the recolouring process. Therefore, the recolouring process defines a sequence of proper colourings.
	
Since $g(n)$ is growing and the fixed point equation $(\ref{definitionplambdaalpha})$ is not a function of $n$, it 
follows that 
\[ \mathbb{E}\left[\vert\hyp'_\text{core}\backslash\hyp'(g(n))\vert\right]\leq g(n),\]
That is, the expected number of vertices not yet stripped by the $g(n)$-th iteration of the stripping process is at most $o(n)$. Further, we know from Lemma \ref{cyclesindepthell} that the expected number of vertices with a cycle in the depth-$g(n)$ neighbourhood is at most $o(n)$. Therefore, it follows from Markov's inequality that \whp\,all but $o(n)$ vertices are $1$-frozen, completing the proof of Theorem \ref{IB_Freezing}(a)(ii).
	
The proof of Theorem \ref{IB_Freezing}(b)(i) follows in a similarly after noting that when $c<c_\text{\rm r}$ the fixed point equation $\rho=(1-e^{-\lambda})^{q-1}$ has only $\rho=0$ as a solution.
\end{proof}

\subsection*{Acknowledgements}

The authors would like to thank the referee for their helpful comments.

\end{document}